\documentclass[12pt, oneside]{scrartcl}
\usepackage{graphicx} 
\usepackage{amssymb}
\usepackage{amsmath}
\usepackage{hyperref}
\usepackage{amsthm}
\usepackage{cleveref}
\usepackage{tikz}
\newtheorem{theorem}{Theorem}[section]
\newtheorem{lemma}{Lemma}[section]

\newtheorem{claim}{Claim}
\newtheorem{Claim}{Claim}

\newtheorem{remark}{Remark}

\newtheorem{conjecture}{Conjecture}
\usepackage{float}
\usepackage{subcaption}

\usepackage{setspace}
\title{2-Distance Coloring of 4-Irregular Planar Graphs}
\author{Sara Al Hajjar}
\date{}

\author{Sara Al Hajjar \footnote{
Kalma Laboratory, Lebanese University, Beirut, Lebanon \\
Univ. Bordeaux, CNRS, Bordeaux INP, LaBRI, UMR 5800, F-33400, Talence,
France. Email: sara.bhajjar@gmail.com
}}
\date{}
\begin{document}

\maketitle
\begin{abstract}
 A $k$-irregular graph is a graph with maximum degree $k$ such that vertices of degree $k$ are not adjacent. A $2$-distance $k$-coloring of a graph is a coloring of the vertices using $k$ colors in which any two vertices at distance at most $2$ receive distinct colors. The $2$-distance chromatic number of $G$, denoted by $\chi_{2}(G)$, is the minimum integer $k$ such that $G$ admits a $2$-distance $k$-coloring. Zhu \cite{zhu} proved that $\chi_2(G)\leq 13$ for planar graphs with maximum degree at most $4$.
 We prove that for a $4$-irregular planar graph $G$, we have $\chi_2(G) \leq 10$. 
\end{abstract}

\section{Introduction}
The discussion in this paper is restricted to finite simple graphs, with the usual notation conventions followed. In a graph $G$, the set of neighbors of a vertex $v$ is denoted by $N_G(v)$. The degree of a vertex $v$, written as $d_G(v)$, represents the number of its neighbors. 
For simplicity, we write $N(v)$ (resp., $d(v)$) instead of $N_G(v)$ (resp., $d_G(v)$). 
We say a graph $G$ is a $k$-irregular graph if $\Delta(G)=k$ and the vertices of degree $k$ in $G$ are not adjacent. A $k$-vertex is a vertex of degree $k$. Moreover, a $k^+$-vertex (resp., $k^-$-vertex) is a vertex of degree at least $k$ (resp., at most $k$). We say $u$ is a $k$-neighbor of a vertex $v$ if $u$ is a $k$-vertex adjacent to $v$. The girth of a graph $G$ is the length of its shortest cycle and is denoted by $g(G)$. A $k$-cycle is a cycle of length $k$. We say two cycles are \textit{adjacent} if they share exactly one common edge. The \textit{square} of a graph $G$, denoted by $G^2$, is the graph obtained from $G$ after adding an edge between every two vertices of distance exactly $2$.
\\

A \textit{planar} graph is a graph that can be drawn with no edge crossings. We call such a drawing a plane graph or a planar embedding of the graph. When a planar graph is drawn with no edge crossings, it divides the
plane into a set of regions, called faces. We denote by $F(G)$ the set of
faces of a planar graph $G$. Each face $f$ is bounded by a closed walk
called the boundary of the face. The degree of a face is the length of its boundary and is denoted by $d(f)$. A face $f$ in a planar graph $G$  is said to be incident with each vertex and edge on its boundary, and two faces are said to be adjacent if their boundaries have an edge in common. A face in a planar graph $G$ is said to be a $k$-face if it is of degree $k$. Moreover, a face in a planar graph $G$ is said to be a $k^+$-face (resp., $k^-$-face) if it is of degree at least $k$ (resp., at most $k$). \\

A $2$-distance $k$-coloring of a graph $G$ is a mapping $\phi : V(G) \to \{1,2,\ldots,k\}$ such that $\phi(v_1) \neq \phi(v_2)$ whenever $d(v_1,v_2) \leq 2$, where $v_1$ and $v_2$ are any two vertices in $G$.  
The \textit{$2$-distance chromatic number} of a graph $G$, denoted by $\chi_2(G)$, is the minimum integer $k$ such that $G$ admits a $2$-distance $k$-coloring. Note that $\chi_2(G)=\chi(G^2)$.
\\ 

Several papers have studied Wegner's conjecture \cite{11} regarding the  $2$-distance chromatic number of planar graphs. Wegner conjectured the following:
\begin{conjecture}(Wegner \cite{11}) If $G$  is a planar graph with maximum degree $\Delta$, then 
$\chi_2(G) \leq 7$ if $\Delta=3$, $\chi_{2} (G)  \leq\Delta + 5$ if $4 \leq \Delta \leq 7$ and $\chi_{2}(G)\leq \lfloor\frac{3\Delta}{2}\rfloor+1$ if $\Delta \ge 8$. \end{conjecture} 
This conjecture remains largely open. Thomassen~\cite{9} proved it for planar graphs with maximum degree $\Delta=3$. Zhu \cite{zhu} proved that $\chi_2(G)\leq 13$ for planar graphs with maximum degree at most $4$.
For larger values of $\Delta$, several upper bounds on $\chi_2(G)$ have been established. In particular, Agnarsson and Halldórsson~\cite{1} proved that
$\chi_2(G)\le \left\lfloor\frac{9\Delta}{5}\right\rfloor+2$
for planar graphs with $\Delta\ge749$, and this was later improved by Borodin \emph{et al.}~\cite{2} to
$\chi_2(G)\le \left\lceil\frac{9\Delta}{5}\right\rceil+1 $
for $\Delta\ge47$. Bousquet~\cite{3} proved that $\chi_2(G)\le2\Delta+7$ for $\Delta\ge9$ and $\chi_2(G)\le21$ for $\Delta\le6$. More recently, Al Hajjar \cite{sh} proved that $\chi_2(G)\leq 3\Delta +2$ for planar graphs with maximum degree $\Delta \ge 6$, hence improving the bound for $6\leq \Delta \leq 8$. For planar graphs with $\Delta\le5$, Zhu and Bu~\cite{12} proved that $\chi_2(G)\le20$, and this bound was subsequently reduced to $19$, $18$, $17$, and finally $16$ by Chen~\cite{4}, Aoki~\cite{6}, Zou \emph{et al.}~\cite{13}, and Zakir Deniz~\cite{5}, respectively. Further results on the $2$-distance chromatic number of planar graphs can be found in \cite{sh1,7,8,10,12}.
\\

We will prove that if $G$ is a $4$-irregular planar graph, then $\chi_2(G) \leq 10$. 
\\

Throughout the paper, we partially color $G^2$ using 10 colors. To extend the 2-distance 10-coloring to the whole graph $G^2$, we repeatedly use the notion of \emph{saving a color} at a vertex $v$. By this, we mean coloring two neighbors of $v$ in $G^2$ while reducing the number of available colors at $v$ by only one. A typical situation arises when $v$ is adjacent in $G^2$ to two nonadjacent vertices $v_1$ and $v_2$, and $|L(v_1)|+|L(v_2)|>|L(v)|$. This inequality guarantees either that $L(v_1)\cap L(v_2)\neq\emptyset$ or that $(L(v_1)\cup L(v_2))\setminus L(v)\neq\emptyset$. In the former case, we assign a common color to $v_1$ and $v_2$. In the latter, we color one of $v_1$ or $v_2$ with a color belonging to $(L(v_1)\cup L(v_2))\setminus L(v)$ and color the other arbitrarily. In either case, only one color is removed from the list of available colors of $v$.

\section{Main Result}
\begin{theorem} \label{thm1} 
If $G$ is a $4$-irregular planar graph, then $\chi_2(G) \leq 10$. \end{theorem}

Assume, for a contradiction, that \Cref{thm1} fails and let $G$ be a counterexample minimizing $|E(G)|$. By minimality of $|E(G)|$, we deduce that $\delta (G)\ge 2$. We will show that $G$ does not exist and hence \Cref{thm1} holds.

\subsection{A Structural Analysis of G}
  The list of available colors of each vertex $v$ is denoted by $L(v)$. 

\begin{remark} \label{r1}
    We always have $\chi_2(H) \leq 10$ for $H=G \backslash\{e\}$ for any edge $e\in E(G)$ by minimality of $|E(G)|$. Thus, to extend the $2$-distance $10$-coloring to the whole graph $G$, we only need to recolor the vertices of the edge $e$. Thus, we will always assume that by minimality of $|E(G)|$, we can greedily color all the vertices in $G^2$ using 10 colors  except two or more adjacent vertices.

\end{remark}
 
 \begin{remark} Consider a partial $2$-distance 10-coloring of $G$. Let $v$ be an uncolored vertex in $G$. Then, $|L(v)|\ge 10-k$ where $k$ is the number of colored neighbors of $v$ in $G^2$. 
 \end{remark}
\begin{lemma} \label{l1}
    $G$ has no cut vertex.
\end{lemma}
\begin{proof}
 Suppose, to the contrary, that $G$ has a cut vertex $v$. Let $C_1, ..., C_t$ be the connected components of $G-v$, $t\ge 2$. Let $G_1 = C_1 \cup \{v\}$ and $G_2 = C_2\cup ... \cup C_t \cup \{v\}$. Since $G_i$ is $4$-irregular and planar, by definition of minimal counterexample, we have $\chi_2(G_i) \leq 10$ for $i=1,2$. By  coloring $v$ by the same color in $G_1$ and $G_2$ and all the neighbors of $v$ by pairwise distinct colors, then combining the colorings of $G_1$ and $G_2$, we obtain  $\chi_2(G) \leq 10$, a contradiction.    
\end{proof}
We deduce that each $k$-vertex is incident to $k$ faces.
\begin{lemma} \label{l2} A $2$-vertex is not adjacent to any $3^-$-vertex. \end{lemma}
\begin{proof} Suppose that there exists a $2$-vertex $v$ adjacent to $3^-$-vertex $u$. By \Cref{r1}, we can greedily color all the vertices in $G^2$ using $10$ colors except $u$ and $v$. Since $|L(u)|\ge 1$ and $|L(v)| \ge 4$, we can greedily color $u$ then $v$ in $G^2$ to obtain $\chi_2(G) \leq 10$, a contradiction.
\end{proof}

\begin{lemma} \label{l3} The distance between the $2$-vertices is at least 3.  \end{lemma}
\begin{proof} 
Suppose, to the contrary, that there exist $2$-vertices $x$ and $y$ such that $d(x,y)\leq 2$. By Lemma \ref{l2}, we deduce that $d(x,y)=2$. 
Let $x_1\in N(x)\cap N(y)$. Greedily color all the vertices in $G^2$ using $10$ colors except $x$ and $x_1$. Then, $|L(x)|\ge 3$ and $|L(x_1)|\ge 1$. Hence, we can greedily color $x_1$ then $x$ in $G^2$ to obtain $\chi_2(G) \leq 10$, a contradiction.
\end{proof}
\begin{lemma} \label{l4} The distance between the $2$-vertices is at least $4$.  \end{lemma}
\begin{proof}
  Suppose, to the contrary, that there exist $2$-vertices $x$ and $y$ such that $d(x,y)\leq 3$. By Lemma \ref{l3}, we deduce that $d(x,y)=3$.  Let $P=xx'y'y$ be an $xy$-path of length $3$. Then, by Lemma \ref{l2}, we deduce that $d(x')=d(y')=4$, a contradiction since $G$ is $4$-irregular. 
\end{proof}
\begin{lemma} \label{l5} $g(G) \ge 4$. 
\end{lemma}
\begin{proof} Suppose that $G$ has a 3-cycle $C$. Set $C=v_1v_2v_3$. Since $G$ is $4$-irregular, at most one vertex of $C$ is of degree $4$. By \Cref{l2}, we deduce that $C$ contains no vertex of  degree $2$. So, $C$ contains at least two vertices of degree $3$. Without loss of generality, suppose that $d(v_1)=d(v_2)=3$. By \Cref{r1}, we can greedily color all the vertices in $G^2$ using 10 colors except $v_1$ and $v_2$. Since  $|L(v_1)| \ge 2$ and $|L(v_2)| \ge 2$, we can greedily color $v_1$ and $v_2$ in $G^2$ to obtain $\chi_2(G) \leq 10$, a contradiction. \end{proof}

\begin{lemma} \label{l6} A $2$-vertex is not contained in a $4$-cycle. \end{lemma}
\begin{proof} Suppose that there exists a $2$-vertex $v$ contained in a $4$-cycle. Let $G'=G-v$. By definition of minimal counterexample, we have $\chi_2(G')\leq 10$. Since $v$ is contained in a $4$-cycle, we deduce that the neighbors of $v$ receive different colors in $G'$. Thus, we only need to color $v$ to extend the $2$-distance $10$-coloring to the whole graph $G$.
Since $d_{G^2}(v)\leq 7$, we can color $v$ by an available color to obtain $\chi_2(G) \leq 10$, a contradiction.
\end{proof}
\begin{lemma} \label{l7}
    Let $v$ be a $4$-vertex contained in a $4$-cycle. Then, $v$ has no $2$-neighbor.
\end{lemma}
\begin{proof}
    Suppose, to the contrary, that $v$ has a $2$-neighbor $u$. Greedily color all the vertices in $G^2$ using 10 colors except $u$ and $v$. Since $|L(v)| \ge 1$ and $|L(u)| \ge 3$, we can greedily color $v$ then $u$ in $G^2$ to obtain $\chi_2(G) \leq 10$, a contradiction.
\end{proof}
\begin{lemma} \label{l8} Every $4$-cycle contains two vertices of degree $3$ and two vertices of degree $4$. \end{lemma}
\begin{proof} Let $C=v_1v_2v_3v_4$ be a $4$-cycle in $G$. Suppose, to the contrary, that $C$ does not contain exactly two $3$-vertices. Since $G$ is 4-irregular and a $2$-vertex is not contained in a $4$-cycle by Lemma \ref{l6}, we deduce that $C$ contains at least three vertices of degree $3$. Without loss of generality, suppose that $d(v_1)=d(v_2)=d(v_3)=3$. By \Cref{r1}, we can greedily color all the vertices in $G^2$ using 10 colors except $v_1$ and $v_2$. Since  $|L(v_1)| \ge 1$ and $|L(v_2)| \ge 2$, we can greedily color $v_1$ then $v_2$ in $G^2$ to obtain $\chi_2(G) \leq 10$, a contradiction. Hence, $C$ contains exactly two vertices of degree $3$. Therefore, $C$ also contains two vertices of degree $4$.
\end{proof} 
We deduce that if a $3$-vertex is incident to a $4$-face, then it has at most one $3$-neighbor.
\begin{lemma} \label{l9} A $3$-vertex is not contained in two adjacent $4$-cycles.\end{lemma}
\begin{proof}
Suppose, to the contrary, that there exists a $3$-vertex $v$ contained in two adjacent $4$-cycles $C_1$ and $C_2$.
    Set $C_1=vv_1v_2v_3$ and $C_2=vv_3v_4v_5$, where $vv_3$ is the common edge shared by $C_1$ and $C_2$. By Lemma \ref{l8}, we deduce that $d(v_1)=d(v_3)=d(v_5)=4$ and $d(v_2)=d(v_4)=3$. We greedily color all vertices in $G^2$ using 10 colors except $V(C_1)\cup V(C_2)$. So, $|L(v)|\ge 5$, $|L(v_1)|\ge 3$, $|L(v_2)|\ge 3$,  $|L(v_3)|\ge 5$, $|L(v_4)|\ge 3$, and $|L(v_5)|\ge 3$. Since $G$ is $4$-irregular and planar, we deduce that either $d(v_1,v_4)=3$ or $d(v_2,v_5)=3$. Without loss of generality, suppose that $d(v_2,v_5)=3$. Since $|L(v_2)|+|L(v_5)|>5$ and  $d(v_2,v_5)=3$, we can color $v_2$ by $c$ and $v_5$ by $c'$ such that $|L(v)\backslash \{c,c'\}|\ge 4$. Color $v_2$ by $c$ and $v_5$ by $c'$ and denote by $L_1$ the new list of available colors of each vertex. 
    
    So, we have $|L_1(v)|\ge 4$, $|L_1(v_1)|\ge 1$, $|L_1(v_3)|\ge 3$, and $|L_1(v_4)|\ge 1$. Note that if $d(v_1,v_4)\leq 2$, then $|L_1(v_1)|\ge 2$ and $|L_1(v_4)|\ge 2$. Hence, we can greedily color in order $v_1$, $v_4$, $v_3$, and $v$ to obtain $\chi_2(G)\leq 10$, a contradiction.
\end{proof}
We deduce that $v$ is incident to at most one $4$-face.
\begin{lemma} \label{l10}
    Let $v$ be a $3$-vertex contained in two $4$-cycles. Then, $v$ has no $3$-neighbor.
\end{lemma}
\begin{proof}
  Suppose that there exists a $3$-vertex $v$ contained in two $4$-cycles $C$ and $C'$ and having a $3$-neighbor. Set $C=vv_1v_2v_3$. By Lemma \ref{l8} and \ref{l9}, we deduce that $C'$ is of the form $v_1v_2v_3x$ where $x\in (N(v_1)\cap N(v_3))\backslash \{v\}$. Greedily color all the vertices in $G^2$ using 10 colors except $v$ and $v_1$. So, $|L(v)|\ge 2$ and $|L(v_1)|\ge 1$. Greedily color $v_1$ then $v$ in $G^2$ to obtain $\chi_2(G)\leq 10$, a contradiction.
\end{proof}
\begin{lemma} \label{l11} Let $u$ and $v$ be two adjacent $3$-vertices such that $u$ and $v$ have two neighbors of degree $3$. Then, the third neighbors of $u$ and $v$ are of degree $4$. 
\end{lemma}
\begin{proof}
    Suppose, to the contrary, that the third neighbor of $v$ is of degree $3$. We greedily color all the vertices in $G^2$ using 10 colors except $u$ and $v$. Since $|L(u)|\ge 1$ and $|L(v)| \ge 2$, we can greedily color $u$ then $v$ to obtain $\chi_2(G) \leq 10$, a contradiction. Similarly, the third neighbor of $u$ is of degree $4$.
\end{proof} 

\begin{figure}[htbp]
    \centering
    \begin{tikzpicture}[
        scale=0.8,
        vertex/.style={circle, fill=black, inner sep=2pt},
        line/.style={thick}
    ]
        \node[vertex, label=left:$u_4$]       (u4) at (0, 4.5) {};
        \node[vertex, label=left:$u_1$]       (u1) at (0, 3.0) {};
        \node[vertex, label=left:$v_1$]       (v1) at (0, 1.5) {};
        \node[vertex, label=below left:$v_4$] (v4) at (0, 0) {};

        \node[vertex, label=right:$u_3$]        (u3) at (2.5, 4.5) {};
        \node[vertex, label=right:$u_2$]        (u2) at (2.5, 3.0) {};
        \node[vertex, label=right:$y$]          (y)  at (2.5, 2.25) {};
        \node[vertex, label=right:$v_2$]        (v2) at (2.5, 1.5) {};
        \node[vertex, label=below right:$v_3$]  (v3) at (2.5, 0) {};

        \draw[line] (u4) -- (u3);
        \draw[line] (u1) -- (u2);
        \draw[line] (v1) -- (v2);
        \draw[line] (v4) -- (v3);

        \draw[line] (u4) -- (u1) -- (v1) -- (v4);
        \draw[line] (u3) -- (u2) -- (y) -- (v2) -- (v3);
    \end{tikzpicture}
    \caption{Illustration of \Cref{l12}}
    \label{f1}
\end{figure}

\begin{lemma} \label{l12} Let $v_1$ and $u_1$ be two vertices of degree $3$ contained in a $4$-face such that $u_1$ and $v_1$ are adjacent. Then, $u_1$ and $v_1$ are not incident to any $5$-face.
\end{lemma}
\begin{proof} Suppose, to the contrary, that $u_1$ and $v_1$ are incident to a $5$-face $f$. By Lemma \ref{l8}, we deduce that $u_1$ and $v_1$ are contained in two disjoint $4$-faces. Let $C=v_1v_2v_3v_4$ be the $4$-face incident to $v_1$ and $C'=u_1u_2u_3u_4$ be the $4$-face incident to $u_1$. So, we have $d(u_2)=d(u_4)=d(v_2)=d(v_4)=4$ and $d(u_3)=d(v_3)=3$. The $5$-face $f$ incident to $v_1$ is of the form $v_1v_2yu_2u_1$ where $y\in N(v_2)\cap N(u_2)$ or of the form $v_1u_1u_4yv_4$ where $y\in N(v_4)\cap N(u_4)$. Without loss of generality, suppose that $f=v_1v_2yu_2u_1$ where $y\in N(v_2)\cap N(u_2)$ (See \Cref{f1}). Since $G$ is $4$-irregular, we have $d(y)=3$. We will study two cases according to $d(v_2,u_3)$. 
\begin{itemize}
    \item \textbf{Case 1:} $d(v_2,u_3)\leq 2$.\\
        By Lemma \ref{l9}, we deduce that $d(u_3,v_2)=2$. Let $x\in N(u_3)\cap N(v_2)$. Since $g(G)\ge 4$ and $G$ is $4$-irregular, we deduce that $x\notin \{u_2,u_4,y\}$ and $d(x)=3$. 

        Suppose that $x\neq v_3$. So, $x\notin V(C)\cup V(C')$. Greedily color all the vertices in $G^2$ using 10 colors except $V(C')\cup V(C)\cup \{x,y\}$. We have $|L(v_1)|\ge 8$, $|L(v_2)|\ge 7$, $|L(v_3)|\ge 4$, $|L(v_4)|\ge 3$, $|L(u_1)|\ge 7$, $|L(u_2)|\ge 6$, $|L(u_3)|\ge 6$, $|L(u_4)|\ge 4$, $|L(x)|\ge 6$, and $|L(y)|\ge 5$. 
Since $G$ is planar with $g(G)\ge 4$, we deduce that $d(y,u_4)=3$. Then, since $|L(y)|+|L(u_4)|>8$, we can color $y$ by $c$ and $u_4$ by $c'$ such that $|L(v_1)\backslash \{c,c'\}|\ge 7$. Color $y$ by $c$ and $u_4$ by $c'$ and call the new list of available colors $L_1$. We have $|L_1(v_1)|\ge 7$, $|L_1(v_2)|\ge 6$, $|L_1(v_3)|\ge 3$, $|L_1(v_4)|\ge 3$, $|L_1(u_1)|\ge 5$, $|L_1(u_2)|\ge 4$, $|L_1(u_3)|\ge 4$, and $|L_1(x)|\ge 4$. Greedily color in order $v_3$, $v_4$, $x$, $u_3$, $u_2$, $v_2$, $u_1$, and $v_1$ in $G^2$ to obtain $\chi_2(G) \leq 10$, a contradiction.  

Suppose now that $x=v_3$. So, $v_3$ and $u_3$ are adjacent. Greedily color the vertices in $G^2$ using 10 colors except $V(C')\cup V(C)\cup \{y\}$. We have $|L(v_1)|\ge 7$, $|L(v_2)|\ge 6$, $|L(v_3)|\ge 7$, $|L(v_4)|\ge 4$, $|L(u_1)|\ge 7$, $|L(u_2)|\ge 6$, $|L(u_3)|\ge 7$, $|L(u_4)|\ge 4$, and $|L(y)|\ge 4$. Hence, since $G$ is planar with $g(G)\ge 4$, we deduce that $d(y,u_4)=3$. Then, since $|L(y)|+|L(u_4)|>7$, we can color $y$ by $c$ and $u_4$ by $c'$ such that $|L(v_1)\backslash \{c,c'\}|\ge 6$. Color $y$ by $c$ and $u_4$ by $c'$ and call the new list of available colors $L_1$. We have $|L_1(v_1)|\ge 6$, $|L_1(v_2)|\ge 5$, $|L_1(v_3)|\ge 5$, $|L_1(v_4)|\ge 4$, $|L_1(u_1)|\ge 5$, $|L_1(u_2)|\ge 4$, and $|L_1(u_3)|\ge 5$. Greedily color in order $v_4$, $v_3$, $u_3$, $u_2$, $v_2$, $u_1$, and $v_1$ in $G^2$ to obtain $\chi_2(G) \leq 10$, a contradiction. 
\item \textbf{Case 2:} $d(v_2,u_3)= 3$.
\\
Greedily color the vertices in $G^2$ using 10 colors except $V(C')\cup V(C)\cup \{y\}$. We have $|L(v_1)|\ge 7$, $|L(v_2)|\ge 5$, $|L(v_3)|\ge 3$, $|L(v_4)|\ge 3$, $|L(u_1)|\ge 7$, $|L(u_2)|\ge 5$, $|L(u_3)|\ge 3$, $|L(u_4)|\ge 3$, and $|L(y)|\ge 4$. 
 Since $|L(u_3)|+|L(v_2)|>7$ and $d(v_2,u_3)= 3$, we can color $v_2$ by $c$ and $u_3$ by $c'$ such that $|L(u_1)\backslash \{c,c'\}|\ge 6$. Color $y$ by $c$ and $u_4$ by $c'$ and call the new list of available colors $L_1$. 
 
 Then, we have $|L_1(v_1)|\ge 6$, $|L_1(v_3)|\ge 2$, $|L_1(v_4)|\ge 2$, $|L_1(u_1)|\ge 6$, $|L_1(u_2)|\ge 3$, $|L_1(u_4)|\ge 2$, and $|L_1(y)|\ge 2$. Greedily color in order $y$, $v_3$, $v_4$, $u_2$, $u_4$, $v_1$, and $u_1$ in $G^2$ to obtain $\chi_2(G) \leq 10$, a contradiction. \qedhere  
\end{itemize}
\end{proof}

Let $f$ be a $5$-face. We say $f$ is a \textit{bad $5$-face} if all the vertices incident to $f$ are of degree $3$. We say $f$ is a \textit{semi-bad $5$-face} if four of its incident vertices are of degree $3$. Otherwise, we say $f$ is a \textit{good $5$-face}. 

Let $v$ be a $3$-vertex. We say $v$ is a $3(1)$-vertex if $v$ is incident to a $4$-face. A $3(1)$-vertex is said to be a \textit{weak $3(1)$-vertex} if it is incident to a $5$-face. A $3(1)$-vertex is said to be a \textit{strong $3(1)$-vertex} if it is not incident to any $5$-face. By \Cref{l12}, we deduce that weak $3(1)$-vertices are not adjacent.
Suppose that $v$ is not incident to any $4$-face. We say $v$ is a \textit{bad $3$-vertex} if $v$ is incident to a bad $5$-face. Similarly, we say $v$ is a \textit{semi-bad $3$-vertex} if $v$ is incident to a semi-bad $5$-face. 
Otherwise, we say $v$ is a \textit{good $3$-vertex}.

We say $v$ is a \textit{bad $2$-vertex} if $v$ is a $2$-vertex incident to a $5$-face and a \textit{good $2$-vertex} otherwise.
\begin{lemma} \label{l13}
    A bad $2$-vertex is incident to exactly one $5$-face.
\end{lemma}
\begin{proof}
    Suppose that there exists a $2$-vertex $v$ incident to two $5$-faces. Let $C=vv_1v_2v_3v_4$ and $C'=vv_4v_5v_6v_1$ be these two $5$-faces incident to $v$. So, we have $d(v_1)=d(v_4)=4$ by \Cref{l2}. Hence, since $G$ is $4$-irregular, $d(v_2)=d(v_3)=d(v_5)=d(v_6)=3$. Greedily color all the vertices in $G^2$ using 10 colors except $V(C)\cup V(C')$. Then, we have $|L(v)|
    \ge 8 $, $|L(v_1)|\ge 5$, $|L(v_2)|\ge 4 $, $|L(v_3)|\ge 4$,
     $|L(v_4)|\ge 5$, $|L(v_5)|\ge 4$, and $|L(v_6)|\ge 4$.
    Since $G$ is planar with $g(G)\ge 4$ , we deduce that either $d(v_2,v_5)=3$ or $d(v_3,v_6)=3$. Without loss of generality, suppose that $d(v_2,v_5)=3$. Since $|L(v_2)|+|L(v_5)|>5$, we can color $v_2$ by $c$ and $v_5$ by $c'$ such that $|L(v_1)\backslash \{c,c'\}|\ge 4$. Color $v_2$ by $c$ and $v_5$ by $c'$ and call the new list $L_1$. So, we have $|L_1(v)|
    \ge 6 $, $|L_1(v_1)|\ge 4$, $|L_1(v_3)|\ge 2$,
     $|L_1(v_4)|\ge 3$, and $|L_1(v_6)|\ge 2$. Greedily color in order $v_3$, $v_6$, $v_4$, $v_1$, and $v$ in $G^2$ to obtain $\chi_2(G)\leq 10$, a contradiction.
\end{proof}
\begin{lemma} \label{l14} A bad $2$-vertex is not incident to any $6$-face. 
\end{lemma}
\begin{proof} 
Suppose that there exists a  bad $2$-vertex $v$ incident to a $6$-face. Let $C=vv_1v_2v_3v_4$ be the $5$-face incident to $v$ and $C'=vv_4v_5v_6v_7v_1$ be the $6$-face incident to $v$.  So, we have $d(v_1)=d(v_4)=4$ and $d(v_2)=d(v_3)=d(v_5)=d(v_7)=3$. Greedily color all the vertices in $G^2$ using 10 colors except $V(C)\cup V(C')$. Then, we have $|L(v)|
    \ge 8 $, $|L(v_1)|\ge 5$, $|L(v_2)|\ge 4 $, $|L(v_3)|\ge 4$,
     $|L(v_4)|\ge 5$, $|L(v_5)|\ge 3$, $|L(v_6)|\ge 2$, and $|L(v_7)|\ge 3$. Note that $v_1$ and $v_4$ are not contained in a $4$-cycle by Lemma \ref{l7}. Therefore, since $G$ is planar with $g(G)\ge 4$, we deduce that either $d(v_2,v_5)=3$ or $d(v_3,v_7)=3$. Without loss of generality, suppose that $d(v_2,v_5)=3$. 
    \begin{itemize}
        \item \textbf{Case 1:} $L(v_2)\cap L(v_5)\neq \emptyset$. \\
        Color $v_2$ and $v_5$ by $c\in L(v_2)\cap L(v_5)$
        and call the new list $L_1$. So, we have $|L_1(v)|
    \ge 7 $, $|L_1(v_1)|\ge 4$, $|L_1(v_3)|\ge 3$,
     $|L_1(v_4)|\ge 4$, $|L(v_6)|\ge 1$, and $|L_1(v_7)|\ge 2$. Greedily color in order $v_6$, $v_7$, $v_3$, $v_1$, $v_4$, and $v$ in $G^2$ to obtain $\chi_2(G)\leq 10$, a contradiction.
     \item \textbf{Case 2:} $L(v_2)\cap L(v_5)= \emptyset$. \\
     Then, $|L(v_2)\cup L(v_5)|>5$. Hence, there exists $c\in L(v_2)\backslash L(v_4)$ or $c\in L(v_5)\backslash L(v_4)$.
     \begin{itemize}
         \item \textbf{Subcase 2.1:} There exists $c\in L(v_2)\backslash L(v_4)$. \\
         Color $v_2$ by $c$ and call the new list $L_1$. Hence, $|L_1(v)|
    \ge 7 $, $|L_1(v_1)|\ge 4$, $|L_1(v_3)|\ge 3$,
     $|L_1(v_4)|\ge 5$, $|L(v_5)|\ge 3$, $|L(v_6)|\ge 2$ and $|L_1(v_7)|\ge 2$. Greedily color in order $v_6$, $v_7$, $v_5$, $v_3$, $v_1$, $v_4$, and $v$ in $G^2$ to obtain $\chi_2(G)\leq 10$, a contradiction.
     \item \textbf{Subcase 2.2:} There exists $c\in L(v_5)\backslash L(v_4)$. \\
         Color $v_5$ by $c$ and call the new list $L_1$. Hence, $|L_1(v)|
    \ge 7 $, $|L_1(v_1)|\ge 5$, $|L(v_2)|\ge 4$, $|L_1(v_3)|\ge 3$,
     $|L_1(v_4)|\ge 5$, $|L(v_6)|\ge 1$ and $|L_1(v_7)|\ge 2$. Greedily color in order $v_6$, $v_7$, $v_3$, $v_2$, $v_1$, $v_4$, and $v$ in $G^2$ to obtain $\chi_2(G)\leq 10$, a contradiction. \qedhere 
     \end{itemize}
    \end{itemize}
\end{proof}
\begin{lemma} \label{l15}
    A bad $2$-vertex is not incident to any $7$-face.
\end{lemma}
\begin{proof}
    Suppose that there exists a bad $2$-vertex incident to a 7-face $f$. Set $f=vv_1v_2v_3v_4v_5v_6$. Let $f'=vv_6v_7v_8v_1$ be the $5$-face incident to $v$. Then, $d(v_1)=d(v_6)=4$. So, $d(v_2)=d(v_5)=d(v_7)=d(v_8)=3$. Moreover, since $G$ is $4$-irregular, $d(v_3)=3$ or $d(v_4)=3$. Without loss of generality, suppose that $d(v_4)=3$.
    Greedily color all the vertices in $G^2$ using 10 colors except the vertices incident to $f$ and $f'$. 

    Then, we have $|L(v)|
    \ge 8$, $|L(v_1)|\ge 5$, $|L(v_2)|\ge 3$, $|L(v_3)|\ge 2$,
     $|L(v_4)|\ge 3$, $|L(v_5)|\ge 4$, $|L(v_6)|\ge 5$, $|L(v_7)|\ge 4$, and $|L(v_8)|\ge 4$. Since $|L(v_6)|>4$, we can color $v_6$ by $c$ such that $|L(v_7)\backslash \{c\}|\ge 4$. Color $v_6$ by $c$ and call the new list of available colors $L_1$.

     We have $|L_1(v)|
    \ge 7$, $|L_1(v_1)|\ge 4$, $|L_1(v_2)|\ge 3$, $|L_1(v_3)|\ge 2$,
     $|L_1(v_4)|\ge 2$, $|L_1(v_5)|\ge 3$, $|L_1(v_7)|\ge 4$, and $|L_1(v_8)|\ge 3$. Greedily color in order $v_4$, $v_3$, $v_5$, $v_2$, $v_1$, $v_8$, $v_7$, and $v$ in $G^2$ to obtain that $\chi_2(G)\leq 10$, a contradiction.
\end{proof}
We deduce that a weak $2$-vertex is incident to an $8^+$-face.
\begin{lemma} \label{l16} Let $f=v_1v_2v_3v_4v_5$ be a semi-bad $5$-face such that $d(v_5)=4$. Then, the third neighbors of $v_1$ and $v_4$ are of degree $4$.
\end{lemma}
\begin{proof} 
We will show that the third neighbor of $v_1$ is of degree $4$. Similarly, we can show that the third neighbor of $v_4$ is of degree $4$.
Suppose, to the contrary, that the third neighbor $u_1$ of $v_1$ is of degree $3$. Since $f$ is semi-bad, we have $d(v_1)=d(v_2)=d(v_3)=d(v_4)=3$. We will study two cases according to $d(u_1,v_3)$.
\begin{itemize}
    \item \textbf{Case 1:} $d(u_1,v_3)=3$.
    \\
    Greedily color all the vertices in $G^2$ using 10 colors except the vertices incident to $f$ and $u_1$.
    So, we have $|L(v_1)|\ge 5$, $|L(v_2)|\ge 5$, $|L(v_3)|\ge 4$, $|L(v_4)|\ge 3$, $|L(v_5)|\ge 3$, and $|L(u_1)|\ge 2$.
    Since $d(u_1,v_3)=3$ and $|L(u_1)|+|L(v_3)|>5$, we can color $u_1$ by $c$ and $v_3$ by $c'$ such that $|L(v_1)\backslash \{c,c'\}|\ge 4$. Color $u_1$ by $c$ and $v_3$ by $c'$ and call the new list of available colors $L_1$.
    
    So, we have $|L_1(v_1)|\ge 4$, $|L_1(v_2)|\ge 3$,
     $|L_1(v_4)|\ge 2$, and $|L_1(v_5)|\ge 1$. Greedily color in order $v_5$, $v_4$, $v_2$, and $v_1$ in $G^2$ to obtain $\chi_2(G)\leq 10$, a contradiction.
     \item \textbf{Case 2:} $d(u_1,v_3)\leq 2$.
     \\
     Since $d(u_1)=3$, we deduce by Lemma \ref{l11} that $d(u_1,v_3)=2$. Let $x\in N(u_1)\cap N(v_3)$. By Lemma \ref{l8}, we deduce that $x\neq v_4$. Moreover, since $g(G)\ge 4$, we deduce that $x$ is not incident to $f$.
     \begin{itemize}
         \item \textbf{Subcase 2.1:} $d(u_1,v_4)=3$.
         \\
         Greedily color all the vertices in $G^2$ using 10 colors except the vertices incident to $f$, $u_1$, and $x$. So, we have $|L(v_1)|\ge 6$, $|L(v_2)|\ge 6$, $|L(v_3)|\ge 6$, $|L(v_4)|\ge 4$, $|L(v_5)|\ge 3$, $|L(u_1)|\ge 4$, and $|L(x)|\ge 3$. 
         Since $d(u_1,v_4)=3$ and $|L(u_1)|+|L(v_4)|>6$, we can color $u_1$ by $c$ and $v_4$ by $c'$ such that $|L(v_1)\backslash \{c,c'\}|\ge 5$. Color $u_1$ by $c$ and $v_4$ by $c'$ and call the new list of available colors $L_1$.
         
         So, we have $|L_1(v_1)|\ge 5$, $|L_1(v_2)|\ge 4$, $|L_1(v_3)|\ge 4$, $|L_1(v_5)|\ge 1$, and $|L_1(x)|\ge 1$. Greedily color in order $x$, $v_5$, $v_3$, $v_2$, and $v_1$ in $G^2$ to obtain $\chi_2(G) \leq 10$, a contradiction.
         \item \textbf{Subcase 2.2:} $d(u_1,v_4)\leq 2$. \\
          By Lemma \ref{l8}, we deduce that $d(u_1,v_4)=2$. Let $y \in N(u_1)\cap N(v_4)$. Since $G$ is planar with $g(G)\ge 4$, we deduce that $y$ is not incident to $f$. Greedily color all the vertices in $G^2$ using 10 colors except the vertices incident to $f$, $u_1$, $x$, and $y$. So, we have $|L(v_1)|\ge 7$, $|L(v_2)|\ge 6$, $|L(v_3)|\ge 7$, $|L(v_4)|\ge 6$, $|L(v_5)|\ge 4$, $|L(u_1)|\ge 6$, $|L(x)|\ge 4$, and $|L(y)|\ge 4$.
          Since $G$ is planar with $g(G)\ge 4$, we deduce that $d(y,v_2)=3$. Since $|L(v_2)|+|L(y)|>7$, we can color $v_2$ by $c$ and $y$ by $c'$ such that $|L(v_1)\backslash \{c,c'\}|\ge 6$. Color $v_2$ by $c$ and $y$ by $c'$ and call the new list of available colors $L_1$. 
          
          So, we have $|L_1(v_1)|\ge 6$, $|L_1(v_3)|\ge 5$, $|L_1(v_4)|\ge 4$, $|L_1(v_5)|\ge 2$, $|L_1(u_1)|\ge 4$, and $|L_1(x)|\ge 2$. Greedily color in order $x$, $v_5$, $u_1$, $v_4$, $v_3$, and $v_1$ in $G^2$ to obtain $\chi_2(G) \leq 10$, a contradiction. \qedhere
     \end{itemize}
\end{itemize}
\end{proof}

\begin{lemma} \label{l17}
    Let $v$ be a bad $3$-vertex. Then, $v$ is incident to two $6^+$-faces.
\end{lemma}
\begin{proof}
 Suppose, to the contrary, that $v$ is incident to two  $5^-$-faces $f$ and $f'$ such that $f'$ is the bad $5$-face. By Lemma \ref{l8}, we deduce that $f$ is a $5$-face. Let $f'=vv_1v_2v_3v_4$ be the bad $5$-face incident to $v$ and $f=vv_4v_5v_6v_7$ be the other $5$-face incident to $v$.  Since $f'$ is a bad $5$-face, $d(v_1)=d(v_2)=d(v_3)=d(v_4)=3$. By \Cref{l11}, we deduce that $d(v_5)=d(v_7)=4$. So, $d(v_6)=3$.
Greedily color all the vertices in $G^2$ using 10 colors except the vertices incident to $f$ and $f'$. Then, we have $|L(v)|\ge 7$, $|L(v_1)|\ge 5 $, $|L(v_2)|\ge 4$, $|L(v_3)|\ge 5$, $|L(v_4)|\ge 7$, $|L(v_5)|\ge 3$, $|L(v_6)|\ge 2$, and $|L(v_7)|\ge 3$. Since $G$ is planar, we deduce that either $d(v_1,v_5)=3$ or $d(v_3,v_7)=3$. Without loss of generality, suppose that $d(v_1,v_5)=3$.
\begin{itemize}
    \item \textbf{Case 1:} $L(v_1)\cap L(v_5)\neq \emptyset$. \\
    Color $v_1$ and $v_5$ by $c\in L(v_1)\cap L(v_5)$ and call the new list of available colors $L_1$. So, $|L_1(v)|\ge 6$, $|L_1(v_2)|\ge 3$, $|L_1(v_3)|\ge 4$, $|L_1(v_4)|\ge 6$, $|L_1(v_6)|\ge 1$, and $|L_1(v_7)|\ge 2$. Greedily color in order $v_6$, $v_7$, $v_2$, $v_3$, $v_4$, and $v$ in $G^2$ to obtain $\chi_2(G) \leq 10$, a contradiction. 
    \item \textbf{Case 2:} $L(v_1)\cap L(v_5)=\emptyset$.
    \\
    Then, $|L(v_1)\cup L(v_5)|\ge 8$. So, there exists $c\in L(v_1)\backslash L(v)$ or $c\in L(v_5)\backslash L(v)$.
    \begin{itemize}
        \item \textbf{Subcase 2.1:} There exists $c\in L(v_1)\backslash L(v)$.
        \\
        Color $v_1$ by $c\in L(v_1)\backslash L(v)$ and call the new list of available $L_1$. So, $|L_1(v)|\ge 7$, $|L_1(v_2)|\ge 3$, $|L_1(v_3)|\ge 4$, $|L_1(v_4)|\ge 6$, $|L_1(v_5)|\ge 3$, $|L_1(v_6)|\ge 2$, and $|L_1(v_7)|\ge 2$. Greedily color in order $v_7$, $v_6$, $v_5$, $v_2$, $v_3$, $v_4$ and $v$ in $G^2$ to obtain $\chi_2(G) \leq 10$, a contradiction.
        \item \textbf{Subcase 2.2:} There exists $c\in L(v_5)\backslash L(v)$. \\
        Color $v_5$ by $c\in L(v_5)\backslash L(v)$ and call the new list of available $L_1$. So, $|L_1(v)|\ge 7$, $|L_1(v_1)|\ge 5$, $|L_1(v_2)|\ge 4$, $|L_1(v_3)|\ge 4$, $|L_1(v_4)|\ge 6$, $|L_1(v_6)|\ge 1$, and $|L_1(v_7)|\ge 2$. Greedily color in order $v_6$, $v_7$, $v_3$, $v_2$, $v_1$, $v_4$ and $v$ in $G^2$ to obtain $\chi_2(G) \leq 10$, a contradiction. \qedhere
    \end{itemize}
\end{itemize}
\end{proof}

\begin{lemma} \label{l18}
     Let $f=v_1v_2v_3v_4v_5$ be a semi-bad $5$-face such that $d(v_5)=4$. Then, $v_2$ and $v_3$ are incident to at least one $6^+$-face. 
\end{lemma}
\begin{proof}
   We will show that $v_2$ is incident to a $6^+$-face. In a similar way, we can show that $v_3$ is also incident to a $6^+$-face. Suppose, to the contrary, that $v_2$ is not incident to any $6^+$-face. Then, by Lemma \ref{l8}, we deduce that $v_2$ is incident to three $5$-faces. Let $f'=v_1u_1u_2u_3v_2$ and $f''=v_2u_3u_4u_5v_3$ be the other two $5$-faces incident to $v_2$ (See \Cref{f2}). By Lemma \ref{l16}, we have $d(u_1)=4$. By Lemma \ref{l11}, we deduce that $d(u_3)=d(u_5)=4$. Hence, since $G$ is $4$-irregular, we have $d(u_2)=d(u_4)=3$. Note that $N(v_1)=\{v_2,v_5,u_1\}$. Hence, since $G$ is $4$-irregular and $g(G) \ge 4$, we deduce that $d(v_1,u_5)=3$. Greedily color all the vertices in $G^2$ using 10 colors except the vertices incident to $f$, $f'$, and $f''$. Hence, we have $|L(v_1)|\ge 6$, $|L(v_2)|\ge 9$, $|L(v_3)|\ge 7$, $|L(v_4)|\ge 4$, $|L(v_5)|\ge 3$, $|L(u_1)|\ge 3$, $|L(u_2)|\ge3 $, $|L(u_3)|\ge 5$, $|L(u_4)|\ge 3$, and $|L(v_5)|\ge 3$.
   \begin{itemize}
       \item \textbf{Case 1:} $L(v_1)\cap L(u_5)\neq \emptyset$. \\
       Color $v_1$ and $u_5$ by $c\in L(v_1)\cap L(u_5)$ and call the new list of available colors $L_1$. Then, 
       $|L_1(v_2)|\ge 8$, $|L_1(v_3)|\ge 6$, $|L_1(v_4)|\ge 3$, $|L_1(v_5)|\ge 2$, $|L_1(u_1)|\ge 2$, $|L_1(u_2)|\ge 2$, $|L_1(u_3)|\ge 4$, and $|L_1(u_4)|\ge 2$. Greedily color in order $u_1$, $u_2$, $u_4$, $u_3$, $v_5$, $v_4$, $v_3$, and $v_2$ in $G^2$ to obtain $\chi_2(G)\leq 10$, a contradiction.

       \item \textbf{Case 2:} $L(v_1)\cap L(u_5)= \emptyset$. \\
       Then, since $|L(v_1)\cup L(u_5)|>7$, we deduce that either there exists $c\in L(v_1)\backslash L(v_3)$ or there exists $c\in L(u_5)\backslash L(v_3)$.
       \begin{itemize}
           \item \textbf{Subcase 2.1:} There exists $c\in L(v_1)\backslash L(v_3)$. \\
            Color $v_1$ by $c\in L(v_1)\backslash L(v_3)$  and call the new list of available colors $L_1$. Then, 
       $|L_1(v_2)|\ge 8$, $|L_1(v_3)|\ge 7$, $|L_1(v_4)|\ge 3$, $|L_1(v_5)|\ge 2$, $|L_1(u_1)|\ge 2$, $|L_1(u_2)|\ge 2$, $|L_1(u_3)|\ge 4$, $|L_1(u_4)|\ge 3$, and $|L_1(u_5)|\ge 3$. Greedily color in order $u_1$, $u_2$, $u_4$, $u_3$, $u_5$, $v_5$, $v_4$, $v_2$, and $v_3$ in $G^2$ to obtain $\chi_2(G)\leq 10$, a contradiction.
  \item \textbf{Subcase 2.1:} There exists $c\in L(u_5)\backslash L(v_3)$. \\
  Color $u_5$ by $c\in L(u_5)\backslash L(v_3)$  and call the new list of available colors $L_1$. Then, 
      $|L_1(v_1)|\ge 6$, $|L_1(v_2)|\ge 8$, $|L_1(v_3)|\ge 7$, $|L_1(v_4)|\ge 3$, $|L_1(v_5)|\ge 3$, $|L_1(u_1)|\ge 3$, $|L_1(u_2)|\ge 3$, $|L_1(u_3)|\ge 4$, and $|L_1(u_4)|\ge 2$. Greedily color in order $u_4$, $u_2$, $u_1$, $u_3$, $v_5$, $v_4$, $v_1$, $v_2$, and $v_3$ in $G^2$ to obtain $\chi_2(G)\leq 10$, a contradiction. \qedhere
       \end{itemize}
   \end{itemize}
\end{proof}

\begin{figure}[htbp]
    \centering
    \begin{subfigure}[b]{0.48\textwidth}
        \centering
        \begin{tikzpicture}[
            scale=0.85,
            vertex/.style={circle, fill=black, inner sep=2pt},
            line/.style={thick}
        ]
            \node[vertex, label=below left:$u_1$]  (u1) at (0, 0) {};
            \node[vertex, label=below:$v_1$]       (v1) at (2.5, 0) {};
            \node[vertex, label=below right:$v_5$] (v5) at (5, 0) {};

            \node[vertex, label=left:$u_2$]        (u2) at (0, 1.5) {};
            \node[vertex, label=right:$v_4$]       (v4) at (5, 1.5) {};

            \node[vertex, label=left:$u_3$]        (u3) at (0, 3) {};
            \node[vertex, label=below right:$v_2$] (v2) at (2.5, 3) {};
            \node[vertex, label=right:$v_3$]       (v3) at (5, 3) {};

            \node[vertex, label=above left:$u_4$]  (u4) at (0, 5) {};
            \node[vertex, label=above right:$u_5$] (u5) at (5, 5) {};

            \draw[line] (u1) -- (v1) -- (v5);
            \draw[line] (u3) -- (v2) -- (v3);
            \draw[line] (u4) -- (u5);

            \draw[line] (u1) -- (u2) -- (u3) -- (u4);
            \draw[line] (v1) -- (v2);
            \draw[line] (v5) -- (v4) -- (v3) -- (u5);
        \end{tikzpicture}
        
        \label{fig:graph1}
    \end{subfigure}
    \hfill
    \begin{subfigure}[b]{0.48\textwidth}
        \centering
        \begin{tikzpicture}[
            scale=0.85,
            vertex/.style={circle, fill=black, inner sep=2pt},
            line/.style={thick}
        ]
            \node[vertex, label=below left:$v_4$]  (v4) at (0, 0) {};
            \node[vertex, label=below:$v_5$]       (v5) at (2.5, 0) {};
            \node[vertex, label=below right:$v_6$] (v6) at (5, 0) {};

            \node[vertex, label=left:$v_7$]        (v7) at (0, 1.5) {};
            \node[vertex, label=right:$v_3$]       (v3) at (5, 1.5) {};

            \node[vertex, label=left:$v_8$]        (v8) at (0, 3) {};
            \node[vertex, label=below right:$v_1$] (v1) at (2.5, 3) {};
            \node[vertex, label=right:$v_2$]       (v2) at (5, 3) {};

            \node[vertex, label=above left:$v_9$]     (v9)  at (0, 5) {};
            \node[vertex, label=above right:$v_{10}$] (v10) at (5, 5) {};

            \draw[line] (v4) -- (v5) -- (v6);
            \draw[line] (v8) -- (v1) -- (v2);
            \draw[line] (v9) -- (v10);

            \draw[line] (v4) -- (v7) -- (v8) -- (v9);
            \draw[line] (v5) -- (v1);
            \draw[line] (v6) -- (v3) -- (v2) -- (v10);
        \end{tikzpicture}

    \end{subfigure}
    \caption{Illustration of \Cref{l18} and \ref{l20}}
    \label{f2}
\end{figure}

\begin{lemma} \label{l19}
     A $3(1)$ vertex is not incident to any bad or semi-bad $5$-faces.
 \end{lemma}
\begin{proof}
    Let $v$ be $3(1)$-vertex incident to a $4$-face $f=vv_1v_2v_3$. Suppose, to the contrary, that $v$ is incident to a semi-bad or bad $5$-face $f'$. Without loss of generality, suppose that $f'=vv_3v_4v_5v_6$. Then, $d(v_4)=d(v_5)=d(v_6)=3$. Moreover, $d(v_2)=3$ and $d(v_1)=d(v_3)=4$. We will study two cases depending on $d(v_2,v_5)$.
    \begin{itemize}
        \item \textbf{Case 1:} $d(v_2,v_5)\leq 2$.\\
        By Lemma \ref{l11}, we deduce that $d(v_2,v_5)=2$. Let $x\in N(v_2)\cap N(v_5)$. Greedily color all the vertices in $G^2$ using $10$ colors except the vertices incident to $f$ and $f'$. So, we have $|L(v)|\ge 6$,   
         $|L(v_1)|\ge 3$,  $|L(v_2)|\ge 4$, $|L(v_3)|\ge 5$, $|L(v_4)|\ge 4$, $|L(v_5)|\ge 5$, and  $|L(v_6)|\ge 5$. By Lemma \ref{l9}, we deduce that $x\neq v_1$. Thus, since $G$ is planar with $g(G)\ge 4$, we deduce that $d(v_1,v_4)=3$. Since $|L(v_1)|+|L(v_4)|>6$, we can color $v_1$ by $c$ and $v_4$ by $c'$ such that $|L(v)\backslash \{c,c'\}|\ge 5$. Color $v_1$ by $c$ and $v_4$ by $c'$ and call the new list of available colors $L_1$. 
         
         So, we have $|L_1(v)|\ge 5$,  $|L_1(v_2)|\ge 2$, $|L_1(v_3)|\ge 3$, $|L_1(v_5)|\ge 4$, and  $|L_1(v_6)|\ge 3$. Greedily color in order $v_2$, $v_3$, $v_6$, $v_5$, and $v$ in $G^2$ to obtain $\chi_2(G)\leq 10 $, a contradiction.
         \item \textbf{Case 2:} $d(v_2,v_5)=3$. \\
         Greedily color all the vertices in $G^2$ using $10$ colors except the vertices incident to $f$ and $f'$. So, we have $|L(v)|\ge 6$,   
         $|L(v_1)|\ge 3$,  $|L(v_2)|\ge 3$, $|L(v_3)|\ge 5$, $|L(v_4)|\ge 4$, $|L(v_5)|\ge 4$, and  $|L(v_6)|\ge 5$. Since $d(v_2,v_5)=3$ and  $|L(v_2)|+|L(v_5)|>6$, we can color $v_2$ by $c$ and $v_5$ by $c'$ such that $|L(v)\backslash \{c,c'\}|\ge 5$. Color $v_2$ by $c$ and $v_5$ by $c'$ and call the new list of available colors $L_1$. 
         
         So, we have $|L_1(v)|\ge 5$,  $|L_1(v_1)|\ge 2$, $|L_1(v_3)|\ge 3$, $|L_1(v_4)|\ge 2$, and  $|L_1(v_6)|\ge 4$. Greedily color in order $v_1$, $v_4$, $v_3$, $v_6$, and $v$ in $G^2$ to obtain $\chi_2(G)\leq 10 $, a contradiction. \qedhere 
    \end{itemize}
\end{proof}

\begin{lemma} \label{l20}
     Let $f=v_1v_2v_3v_4v_5$ be a semi-bad $5$-face such that $d(v_5)=4$. Then, $v_1$ and $v_4$ are incident to at least one $6^+$-face.
\end{lemma}
\begin{proof}
    We will show that $v_1$ is incident to a $6^+$-face. In a similar way, we can show that $v_4$ is also incident to a $6^+$-face.
   Suppose, to the contrary, that $v_1$ is not incident to any $6^+$-face. By \Cref{l19}, we deduce that $v_1$ is not incident to any $4$-face.

   Let $f'=v_1v_5v_6v_7v_8$ and $f''=v_1v_8v_9v_{10}v_2$ be the other two $5$-faces incident to $v_1$ (See \Cref{f2}). By \Cref{l11}, we deduce that $d(v_{10})=4$. By \Cref{l16}, we have $d(v_8)=4$. Thus, since $G$ is $4$-irregular, $d(v_6)=d(v_7)=d(v_9)=3$.
   We will study two cases according to $d(v_5,v_{10})$.
   \begin{itemize}
       \item \textbf{Case 1:} $d(v_5,v_{10})\leq 2$. \\
       Greedily color all the vertices in $G^2$ using 10 colors except the vertices incident to $f$, $f'$, and $f''$. Then, we have   
    $|L(v_1)|\ge 8$, $|L(v_2)|\ge 7$,  $|L(v_3)|\ge 5$, $|L(v_4)|\ge 4$, $|L(v_5)|\ge 6$, $|L(v_6)|\ge 4$, $|L(v_7)|\ge 4$, $|L(v_8)|\ge 5$, $|L(v_9)|\ge 3 $, $|L(v_{10})|\ge 4$.
    Since $G$ is planar with $g(G)\ge 4$, we deduce that $d(v_4,v_8)= 3$. Then, since $|L(v_4)|+|L(v_8)|>7$, we can color $v_4$ by $c$ and $v_8$ by $c'$ such that $|L(v_2)\backslash \{c,c'\}|\ge 6$. Color $v_4$ by $c$ and $v_8$ by $c'$
    and call the new list of available colors $L_1$.
    
    So, we have    
    $|L_1(v_1)|\ge 6$,  $|L_1(v_2)|\ge 6$, $|L_1(v_3)|\ge 4$, $|L_1(v_5)|\ge 4$, $|L_1(v_6)|\ge 2$ $|L_1(v_7)|\ge 3$, $|L_1(v_9)|\ge 2$, $|L_1(v_{10})|\ge 3$. Greedily color in order $v_9$, $v_6$, $v_7$, $v_{10}$, $v_5$, $v_1$, $v_3$, and $v_2$ in $G^2$ to obtain $\chi_2(G)\leq 10$, a contradiction.
       \item \textbf{Case 2:} $d(v_5,v_{10})= 3$. \\
       Greedily color all the vertices in $G^2$ using 10 colors except the vertices incident to $f$, $f'$, and $f''$. Then, we have   
    $|L(v_1)|\ge 8$, $|L(v_2)|\ge 7$,  $|L(v_3)|\ge 5$, $|L(v_4)|\ge 4$, $|L(v_5)|\ge 5$, $|L(v_6)|\ge 4$, $|L(v_7)|\ge 4$, $|L(v_8)|\ge 5$, $|L(v_9)|\ge 3 $, $|L(v_{10})|\ge 3$.
    Since $d(v_5,v_{10})=3$ and $|L(v_5)|+|L(v_{10})|>7$, we can color $v_5$ by $c$ and $v_{10}$ by $c'$ such that $|L(v_2))\backslash \{c,c'\}|\ge 6$. Color $v_5$ by $c$ and $v_{10}$ by $c'$
    and call the new list of available colors $L_1$.
    
    So, we have    
    $|L_1(v_1)|\ge 6$,  $|L_1(v_2)|\ge 6$, $|L_1(v_3)|\ge 3$, $|L_1(v_4)|\ge 3$, $|L_1(v_6)|\ge 3$, $|L_1(v_7)|\ge 3$, $|L_1(v_8)|\ge 3$, and $|L_1(v_9)|\ge 2$. Greedily color in order $v_9$, $v_8$, $v_7$, $v_6$, $v_4$, $v_1$, $v_3$, and $v_2$ in $G^2$ to obtain $\chi_2(G)\leq 10$, a contradiction. \qedhere
   \end{itemize}
\end{proof}
By \Cref{l18} and \ref{l20}, we deduce that every semi-bad $3$-vertex is incident to a $6^+$-face.
\begin{lemma} \label{l21}
    Let $v$ be a $3(1)$-vertex incident to two $5$-faces. Then, all neighbors of $v$ are of degree $4$.
\end{lemma}
\begin{proof}
Suppose, to the contrary, that $v$ has a $3$-neighbor.
Let $f=vv_1v_2v_3$ be the $4$-face incident to $v$. Let $f'=vv_3v_4v_5v_6$ and $f''=vv_6v_7v_8v_1$ be the $5$-faces incident to $v$ (See \Cref{f3}). By Lemma \ref{l8}, we deduce that $d(v_1)=d(v_3)=4$ and $d(v_2)=3$. Hence, $v_6$ is the $3$-neighbor of $v_1$. Since $G$ is $4$-irregular, we have $d(v_4)=d(v_8)=3$.
\begin{claim} \label{c1}
    $d(v_2,v_6)=3$.
\end{claim}
\begin{proof}
Suppose, to the contrary, that $d(v_2,v_6)\leq 2$.
Since $N(v_6)=\{v,v_5,v_7\}$, we deduce that $v_2\notin N(v_6)$. Hence $d(v_2,v_6)=2$. Note than $v\notin N(v_2)$ since $g(G)\ge 4$. Then, either $v_5\in N(v_2)\cap N(v_6)$ or $v_7\in N(v_2)\cap N(v_6)$. Without loss of generality, suppose that $v_5\in N(v_2)\cap N(v_6)$.
Then, $v_2$ is contained in two adjacent $4$-cycles $vv_1v_2v_3$ and $v_2v_3v_4v_5$, a contradiction by Lemma \ref{l9}.
\end{proof}
\begin{claim} \label{c2}
    $d(v_3,v_8)=3$.
\end{claim}
\begin{proof}
Suppose, to the contrary, that $d(v_3,v_8)\leq 2$. Suppose that $d(v_3,v_8)=1$. Then $v$ is contained in two $4$-cycles and has a $3$-neighbor, a contradiction by \Cref{l10}. Hence, $d(v_3,v_8)=2$.
    Let $x\in N(v_3)\cap N(v_8)$. Since $G$ is planar and $4$-irregular with $g(G)\ge 4$, we deduce that $x\notin \{v_1,v_2,v_5,v_7\}$. Moreover, we have $d(x)=3$.
\begin{itemize}
    \item \textbf{Case 1:} $x\neq v_4$. \\
    Then, $x$ is not incident to $f$, $f'$, or $f''$.
    Greedily color all the vertices in $G^2$ using 10 colors except $x$ and the vertices incident to $f$, $f'$, and $f''$. Then, we have $|L(v)|\ge 9$,   
    $|L(v_1)|\ge 6$, $|L(v_2)|\ge 5$,  $|L(v_3)|\ge 7$, $|L(v_4)|\ge 4$, $|L(v_5)|\ge 3$, $|L(v_6)|\ge 6$, $|L(v_7)|\ge 4$, $|L(v_8)|\ge 6$, and $|L(x)|\ge 6$.
    Since $d(v_2,v_6)=3$ by \Cref{c1} and $|L(v_2)|+|L(v_6)|>9$, we can color $v_2$ by $c$ and $v_6$ by $c'$ such that $|L(v)\backslash \{c,c'\}|\ge 8$. Color $v_2$ by $c$ and $v_6$ by $c'$
    and call the new list of available colors $L_1$. 
    
    So, we have $|L_1(v)|\ge 8$,   
    $|L_1(v_1)|\ge 4$,  $|L_1(v_3)|\ge 5$, $|L_1(v_4)|\ge 2$, $|L_1(v_5)|\ge 2$, $|L_1(v_7)|\ge 3$, $|L_1(v_8)|\ge 4$, and $|L_1(x)|\ge 5$. Since $G$ is planar, we deduce that $d(v_1,v_5)=3$. Since $|L_1(v_1)|+|L_1(v_5)|>5$, we can color $v_1$ by $c_1$ and $v_5$ by $c_2$ such that $|L_1(v_3)\backslash \{c_1,c_2\}|\ge 4$.  Color $v_1$ by $c_1$ and $v_5$ by $c_2$ and call the new list of available \mbox{colors $L_2$.}  
    
    So, we have $|L_2(v)|\ge 6$,   $|L_2(v_3)|\ge 4$, $|L_2(v_4)|\ge 1$, $|L_2(v_7)|\ge 1$, $|L_2(v_8)|\ge 3$, and $|L_2(x)|\ge 4$. Greedily color in order $v_4$, $v_7$, $v_8$, $x$, $v_3$, and $v$ in $G^2$ to obtain $\chi_2(G)\leq 10$, a contradiction.
    \item \textbf{Case 2:} $x=v_4$. \\
    So, $v_4$ and $v_8$ are adjacent.
Greedily color all the vertices in $G^2$ using $10$ colors except the vertices incident to $f$, $f'$, and $f''$. Then, we have $|L(v)|\ge 8$,   
    $|L(v_1)|\ge 6$, $|L(v_2)|\ge 4$,  $|L(v_3)|\ge 6$, $|L(v_4)|\ge 7$, $|L(v_5)|\ge 4$, $|L(v_6)|\ge 6$, $|L(v_7)|\ge 4$, and $|L(v_8)|\ge 7$.
    Since $d(v_2,v_6)=3$ by \Cref{c1} and $|L(v_2)|+|L(v_6)|>8$, we can color $v_2$ by $c$ and $v_6$ by $c'$ such that $|L(v)\backslash \{c,c'\}|\ge 7$. Color $v_2$ by $c$ and $v_6$ by $c'$
    and call the new list of available colors $L_1$.
    
    So, we have $|L_1(v)|\ge 7$,   
    $|L_1(v_1)|\ge 4$,  $|L_1(v_3)|\ge 4$, $|L_1(v_4)|\ge 5$, $|L_1(v_5)|\ge 3$, $|L_1(v_7)|\ge 3$, and $|L_1(v_8)|\ge 5$.  Note that since $N(v_4)=\{v_3,v_5,v_8\}$, we deduce that $v_1\notin N(v_4)$. Similarly, $v_5\notin N(v_8)$.
    Hence, since $G$ is planar, we deduce that $d(v_1,v_5)=3$. Since $|L_1(v_1)|+|L_1(v_5)|>4$, we can color $v_1$ by $c_1$ and $v_5$ by $c_2$ such that $|L_1(v_3)\backslash \{c_1,c_2\}|\ge 3$.  Color $v_1$ by $c_1$ and $v_5$ by $c_2$ and call the new list of available colors $L_2$.  
    
    So, we have $|L_2(v)|\ge 5$, $|L_2(v_3)|\ge 3$, $|L_2(v_4)|\ge 3$, $|L_2(v_7)|\ge 1$, and $|L_2(v_8)|\ge 3$. Greedily color in order $v_7$, $v_8$, $v_4$, $v_3$, and $v$ in $G^2$ to obtain $\chi_2(G)\leq 10$, a contradiction. \qedhere
    \end{itemize}
    \end{proof}
 Now, greedily color all the vertices in $G^2$ using $10$ colors except the vertices incident to $f$, $f'$, and $f''$.  Then, we have $|L(v)|\ge 8$,   
    $|L(v_1)|\ge 5$, $|L(v_2)|\ge 4$,  $|L(v_3)|\ge 5$, $|L(v_4)|\ge 3$, $|L(v_5)|\ge 3$, $|L(v_6)|\ge 6$, $|L(v_7)|\ge 3$, and $|L(v_8)|\ge 3$. 
\begin{claim} \label{c3}
  $L(v_2)\subset L(v)$. 
\end{claim}
\begin{proof}
Suppose, to the contrary, that $L(v_2)\nsubseteq L(v)$. Then, there exists $c\in L(v_2)\backslash L(v)$. Color $v_2$ by $c$ and call the new list of available colors $L_1$.

So, we have $|L_1(v)|\ge 8$,   
    $|L_1(v_1)|\ge 4$,  $|L_1(v_3)|\ge 4$, $|L_1(v_4)|\ge 2$, $|L_1(v_5)|\ge 3$, $|L_1(v_6)|\ge 6$, $|L_1(v_7)|\ge 3$, and $|L_1(v_8)|\ge 2$. Since $d(v_3,v_8)\ge 3$ by \Cref{c2} and $|L_1(v_3)|+|L_1(v_8)|>4$, we can color $v_3$ by $c_1$ and $v_8$ by $c_2$ such that $|L_1(v_1)\backslash \{c_1,c_2\}|\ge 3$. Color $v_3$ by $c_1$ and $v_8$ by $c_2$ and call the new list of available colors $L_2$.
    
    So, we have $|L_2(v)|\ge 6$, $|L_2(v_1)|\ge 3$, $|L_2(v_4)|\ge 1$, $|L_2(v_5)|\ge 2$, $|L_2(v_6)|\ge 4$, and $|L_2(v_7)|\ge 2$. Greedily color in order $v_4$, $v_5$, $v_7$, $v_6$, $v_1$, and $v$ in $G^2$ to obtain $\chi_2(G)\leq 10$, a contradiction. 
    \end{proof}
\begin{claim} \label{c4}
  $L(v_1)\subset L(v)$ and $L(v_3)\subset L(v)$. 
\end{claim}
\begin{proof}
We will show that $L(v_1)\subset L(v)$. Similarly, we can show that $L(v_3)\subset L(v)$.
Suppose, to the contrary, that $L(v_1)\nsubseteq L(v)$.  Then, there exists $c\in L(v_1)\backslash L(v)$. Color $v_1$ by $c$ and call the new list of available colors $L_1$. Since $L(v_2)\subset L(v)$ by \Cref{c3}, we deduce that $c\notin L(v_2)$.

So, we have $|L_1(v)|\ge 8$,   
    $|L_1(v_2)|\ge 4$,  $|L_1(v_3)|\ge 4$, $|L_1(v_4)|\ge 3$, $|L_1(v_5)|\ge 3$, $|L_1(v_6)|\ge 5$, $|L_1(v_7)|\ge 2$, and $|L_1(v_8)|\ge 2$. Greedily color in order $v_8$, $v_7$, $v_5$, $v_4$, $v_6$, $v_3$, $v_2$, and $v$ in $G^2$ to obtain $\chi_2(G)\leq 10$, a contradiction. 
    \end{proof}
\begin{claim} \label{c5}
    $L(v_2)\cap L(v_6)=\emptyset$.
\end{claim}
\begin{proof}
Suppose that there exists $c\in L(v_2)\cap L(v_6)$. Recall that $d(v_2,v_6)=3$ by \Cref{c1}. Color $v_2$ and $v_6$ by $c$ and the new list of available colors $L_1$.
So, we have $|L_1(v)|\ge 7$,   
    $|L_1(v_1)|\ge 4$,  $|L_1(v_3)|\ge 4$, $|L_1(v_4)|\ge 2$, $|L_1(v_5)|\ge 3$, $|L_1(v_7)|\ge 2$, and $|L_1(v_8)|\ge 2$. Greedily color in order $v_8$, $v_7$, $v_5$, $v_4$, $v_3$, $v_1$, and $v$ in $G^2$ to obtain $\chi_2(G)\leq 10$, a contradiction. 
\end{proof}
\textbf{Final Step:} Since $L(v_2)\cap L(v_6)=\emptyset$ by \Cref{c5}, $|L(v_2)|+|L(v_6)|>8$, and $L(v_2)\subset L(v)$, we deduce that there exists $c \in L(v_6)\backslash L(v)$. By Claim \ref{c4}, we deduce that $c\notin L(v_1)\cup L(v_3)$. Color $v_6$ by $c$ and the call the new list of available colors $L_1$.
So, we have $|L_1(v)|\ge 8$,   
    $|L_1(v_1)|\ge 5$,  $|L_1(v_2)|\ge 4$, $|L_1(v_3)|\ge 5$, $|L_1(v_4)|\ge 2$, $|L_1(v_5)|\ge 2$, $|L_1(v_7)|\ge 2$, and $|L_1(v_8)|\ge 2$. Greedily color in order $v_8$, $v_7$, $v_5$, $v_4$, $v_3$, $v_2$, $v_1$, and $v$ in $G^2$ to obtain $\chi_2(G)\leq 10$, a contradiction. 
\end{proof}

\begin{figure}[htbp]
    \centering
    \begin{tikzpicture}[
        scale=0.8,
        vertex/.style={circle, fill=black, inner sep=2pt},
        line/.style={thick}
    ]
        \node[vertex, label=below left:$v_4$]  (v4) at (0, 0) {};
        \node[vertex, label=below:$v_3$]       (v3) at (4, 0) {};
        \node[vertex, label=below right:$v_2$] (v2) at (6, 0) {};

        \node[vertex, label=above left:$v_5$]  (v5) at (0, 2) {};
        \node[vertex, label=above left:$v_6$]  (v6) at (2, 2) {};
        \node[vertex, label=below right:$v$]   (v)  at (4, 2) {};
        \node[vertex, label=right:$v_1$]       (v1) at (6, 2) {};

        \node[vertex, label=above left:$v_7$]  (v7) at (2, 4) {};
        \node[vertex, label=above right:$v_8$] (v8) at (6, 4) {};

        \draw[line] (v4) -- (v3) -- (v2);
        \draw[line] (v5) -- (v6) -- (v) -- (v1);
        \draw[line] (v7) -- (v8);

        \draw[line] (v4) -- (v5);
        \draw[line] (v6) -- (v7);
        \draw[line] (v3) -- (v);
        \draw[line] (v2) -- (v1) -- (v8);
    \end{tikzpicture}
    \caption{Illustration of \Cref{l21} and \ref{l22}}
    \label{f3}
\end{figure}

 \begin{lemma} \label{l22}
    Let $v$ be a $3(1)$-vertex. Then, $v$ is incident to at most one $5$-face.
\end{lemma}
\begin{proof}
  Suppose, to the contrary, that $v$ is incident to two $5$-faces.
Let $f=vv_1v_2v_3$ be the $4$-face incident to $v$. Let $f'=vv_3v_4v_5v_6$ and $f''=vv_6v_7v_8v_1$ be the $5$-faces incident to $v$ (See \Cref{f3}). By Lemma \ref{l8}, we have $d(v_1)=d(v_3)=4$ and $d(v_2)=3$. By Lemma \ref{l21}, $d(v_6)=4$. Since $G$ is $4$-irregular, we have $d(v_4)=d(v_5)=d(v_7)=d(v_8)=3$.  
\begin{Claim} \label{cc1}
    $d(v_2,v_7)=3$.
\end{Claim}
\begin{proof}
    Suppose, to the contrary, that $d(v_2,v_7)\leq 2$. By Lemma \ref{l9}, we deduce that $d(v_2,v_7)=2$. Let $x\in N(v_2)\cap N(v_7)$. Since $g(G)\ge 4$, we deduce that $x\notin \{v_1,v_8\}$. Moreover, $x\neq v_3$ since otherwise $v$ is contained in two adjacent $4$-cycles $vv_1v_2v_3$ and $vv_6v_7v_3$, a contradiction by Lemma \ref{l9}. 
    
    We will show that $x\neq v_6$.
    Suppose that $x=v_6$. Then, $v$, $v_1$, $v_2$, $v_3$, and $v_6$ are contained in two $4$-cycles. Greedily color all the vertices in $G^2$ using 10 colors except the vertices incident to $f$, $f'$, and $f''$. So, we have $|L(v)|\ge 8$, $|L(v_1)|\ge 6$, $|L(v_2)|\ge 8$, $|L(v_3)|\ge 6$, $|L(v_4)|\ge 4$, $|L(v_5)|\ge 5$, $|L(v_6)|\ge 8$, $|L(v_7)|\ge 5 $, and $|L(v_8)|\ge 4$. Since $G$ is planar with $g(G) \ge 4$, we deduce that $d(v_3,v_8)=3$. Since $|L(v_3)|+|L(v_8)|>8$, we can color $v_3$ by $c$ and $v_8$ by $c'$ such that $|L(v)\backslash \{c,c'\}|\ge 7$. Color $v_3$ by $c$ and $v_8$ by $c'$ and call the new list of available colors $L_1$. So, we have $|L_1(v)|\ge 7$, $|L_1(v_1)|\ge 4$, $|L_1(v_2)|\ge 6$, $|L_1(v_4)|\ge 3$, $|L_1(v_5)|\ge 4$, $|L_1(v_6)|\ge 6$, and $|L_1(v_7)|\ge 4 $. Greedily color in order $v_4$, $v_5$, $v_7$, $v_1$, $v_6$, $v_2$, and $v$ in $G^2$ to obtain $\chi_2(G)\leq 10$, a contradiction. Therefore, $x$ is not incident to $f$, $f'$, or $f''$.

We will show that $d(x)=4$. Suppose, to the contrary, that $d(x)=3$. Greedily color all the vertices in $G^2$ using 10 colors except $x$ and the vertices incident to $f$, $f'$, and $f''$. So, we have $|L(v)|\ge 7$, $|L(v_1)|\ge 6$, $|L(v_2)|\ge 7$, $|L(v_3)|\ge 6$, $|L(v_4)|\ge 4$, $|L(v_5)|\ge 4$, $|L(v_6)|\ge 6$, $|L(v_7)|\ge 7$, $|L(v_8)|\ge 5$, and $|L(x)|\ge 6$. By Lemma \ref{l8}, we deduce that $v_2 \notin N(v_5)$. Hence, since $G$ is planar with $g(G)\ge 4$, we deduce that $d(v_1,v_5)=3$. Since $|L(v_1)|+|L(v_5)|>7$, we can color $v_1$ by $c$ and $v_5$ by $c'$ such that $|L(v_7)\backslash \{c,c'\}|\ge 6$. Color $v_1$ by $c$ and $v_5$ by $c'$ and call the new list of available colors $L_1$. So, we have $|L_1(v)|\ge 5$, $|L_1(v_2)|\ge 6$, $|L_1(v_3)|\ge 4$, $|L_1(v_4)|\ge 3$, $|L_1(v_6)|\ge 4$, $|L_1(v_7)|\ge 6$, $|L_1(v_8)|\ge 4$, and $|L_1(x)|\ge 5$. Greedily color in order $v_4$, $v_8$, $v_6$, $v_3$, $v$, $v_2$, $x$, and $v_7$ to obtain that $\chi_2(G)\leq 10$, a contradiction. So, $d(x)=4$.
\\

Now, we will study two cases depending on $d(x,v_5)$.
\begin{itemize}
    \item \textbf{Case 1:} $d(x,v_5)=3$. \\
    Greedily color all the vertices in $G^2$ using 10 colors except $x$ and the vertices incident to $f$, $f'$, and $f''$. So, we have $|L(v)|\ge 7$, $|L(v_1)|\ge 6$, $|L(v_2)|\ge 6$, $|L(v_3)|\ge 6$, $|L(v_4)|\ge 4$, $|L(v_5)|\ge 4$, $|L(v_6)|\ge 6$, $|L(v_7)|\ge 6$, $|L(v_8)|\ge 5$, and $|L(x)|\ge 4$. Since $d(x,v_5)=3$ and $|L(x)|+|L(v_5)|>6$, we can  color $x$ by $c$ and $v_5$ by $c'$ such that $|L(v_7)\backslash \{c,c'\}|\ge 5$. Color $x$ by $c$ and $v_5$ by $c'$ and call the new list of available colors $L_1$.
    
    So, we have $|L_1(v)|\ge 6$, $|L_1(v_1)|\ge 5$, $|L_1(v_2)|\ge 5$, $|L_1(v_3)|\ge 4$, $|L_1(v_4)|\ge 3$, $|L_1(v_6)|\ge 4$, $|L_1(v_7)|\ge 5$, and $|L_1(v_8)|\ge 4$. Since $G$ is planar with $g(G)\ge 4$, we deduce that $d(v_4,v_8)\ge 3$. Since $|L_1(v_4)|+|L_1(v_8)|>6$, we can color $v_4$ by $c_1$ and $v_8$ by $c_2$ such that $|L_1(v)\backslash \{c_1,c_2\}|\ge 5$. Color $v_4$ by $c_1$ and $v_8$ by $c_2$ and call the new list of available colors $L_2$.  

    So, we have $|L_2(v)|\ge 5$, $|L_2(v_1)|\ge 4$, $|L_2(v_2)|\ge 3$, $|L_2(v_3)|\ge 3$, $|L_2(v_6)|\ge 2$, and $|L_2(v_7)|\ge 4$. Since $g(G)\ge 4$, we deduce that $d(v_2,v_6)=3$. Since $|L_2(v_2)|+|L_2(v_6)|>4$, we can color $v_2$ by $c_3$ and $v_6$ by $c_4$ such that $|L_2(v_7)\backslash \{c_3,c_4\}|\ge 3$. Color $v_4$ by $c_3$ and $v_8$ by $c_4$ and call the new list of available colors $L_3$.  
    
    So, we have $|L_3(v)|\ge 3$, $|L_3(v_1)|\ge 2$, $|L_3(v_3)|\ge 1$, and $|L_3(v_7)|\ge 3$. Greedily color in order $v_3$, $v_1$, $v$, and $v_7$ in $G^2$ to obtain $\chi_2(G)\leq 10$, a contradiction.
    \item \textbf{Case 2:} $d(x,v_5)\leq 2$.
    \\
    Note that $v_4\notin N(x)$ since otherwise $v_2$ and $v_3$ are contained in two adjacent $4$-cycles, a contradiction by \Cref{l9}.
    So, either $v_5$ has a second $3$-neighbor or $v_5$ is contained in a $4$-cycle. Moreover, $v_2\notin N(v_5)$ since $g(G)\ge 4$. By \Cref{l8}, we deduce that $v_2\notin N(v_5)$.
    
    Greedily color all the vertices in $G^2$ using 10 colors except $x$ and the vertices incident to $f$, $f'$, and $f''$. So, we have $|L(v)|\ge 7$, $|L(v_1)|\ge 6$, $|L(v_2)|\ge 6$, $|L(v_3)|\ge 6$, $|L(v_4)|\ge 4$, $|L(v_5)|\ge 6$, $|L(v_6)|\ge 6$, $|L(v_7)|\ge 6$, $|L(v_8)|\ge 5$, and $|L(x)|\ge 5$. By \Cref{l9}, we deduce that $x\notin N(v_4)$.
Hence, since $G$ is planar, we deduce that $d(v_4,v_7)=3$. Since $|L(v_4)|+|L(v_7)|>6$, we can color $v_4$ by $c$ and $v_7$ by $c'$ such that $|L(v_5)\backslash \{c,c'\}|\ge 5$. Color $v_4$ by $c$ and $v_7$ by $c'$ and call the new list of available colors $L_1$.
    
    So, we have $|L_1(v)|\ge 5$, $|L_1(v_1)|\ge 5$, $|L_1(v_2)|\ge 4$, $|L_1(v_3)|\ge 5$, $|L_1(v_5)|\ge 5$, $|L_1(v_6)|\ge 4$, $|L_1(v_8)|\ge 4$, and $|L_1(x)|\ge 4$. Since $g(G)\ge 4$, we have $d(v_2,v_6)=3$. Since $|L_1(v_2)|+|L_1(v_6)|>5$, we can color $v_2$ by $c_1$ and $v_6$ by $c_2$ such that $|L_1(v)\backslash \{c_1,c_2\}|\ge 4$. Color $v_2$ by $c_1$ and $v_6$ by $c_2$ and call the new list of available colors $L_2$. 

    So, we have $|L_2(v)|\ge 4$, $|L_2(v_1)|\ge 3$, $|L_2(v_3)|\ge 3$, $|L_2(v_5)|\ge 4$, $|L_2(v_8)|\ge 2$, and $|L_2(x)|\ge 2$. Greedily color in order $x$, $v_8$, $v_1$, $v_3$, $v$, and $v_5$ in $G^2$ to obtain $\chi_2(G)\leq 10$, a contradiction. \qedhere
\end{itemize}
\end{proof}
\begin{Claim} \label{cc2}
    $d(v_4,v_8)\ge 3$.
\end{Claim}
\begin{proof}
    Suppose, to the contrary, that $d(v_4,v_8)\leq 2$. We will show that $d(v_4,v_8)=2$. Suppose that $d(v_4,v_8)=1$. Greedily color all the vertices in $G^2$ using 10 colors the vertices incident to $f$, $f'$, and $f''$. So, we have $|L(v)|\ge 7$, $|L(v_1)|\ge 6$, $|L(v_2)|\ge 4$, $|L(v_3)|\ge 6$, $|L(v_4)|\ge 8$, $|L(v_5)|\ge 5$, $|L(v_6)|\ge 5$, $|L(v_7)|\ge 5$, and $|L(v_8)|\ge 8$.
    By \Cref{cc1}, $d(v_2,v_7)=3$. Since $|L(v_2)|+|L(v_7)|>8$, we can color $v_2$ by $c$ and $v_7$ by $c'$ such that $|L(v_8)\backslash \{c,c'\}|\ge 7$. Color $v_2$ by $c$ and $v_7$ by $c'$ and call the new list of available colors $L_1$.
So, we have $|L_1(v)|\ge 5$, $|L_1(v_1)|\ge 4$, $|L_1(v_3)|\ge 5$, $|L_1(v_4)|\ge 6$, $|L_1(v_5)|\ge 4$, $|L_1(v_6)|\ge 4$, and $|L_1(v_8)|\ge 7$.
   Greedily color in order $v_6$, $v_5$, $v_1$, $v_3$, $v$, $v_4$, and $v_8$ in $G^2$ to obtain $\chi_2(G)\leq 10$, a contradiction. Hence, $d(v_4,v_8)=2$.

   Let $x\in N(v_4)\cap N(v_8)$. By Lemma \ref{l8}, we deduce that $x\neq \{v_5,v_7\}$. 
   Suppose that $x=v_1$. Then, $v_4$ is a $3$-vertex contained in two $4$-cycles $v_1v_2v_3v_4$ and $vv_1v_2v_3$ and has a $3$-neighbor $v_5$, a contradiction by Lemma \ref{l10}. Similarly, $x\neq v_3$. Thus, $x$ is not incident to $f$, $f'$, or $f''$. 

Greedily color all the vertices in $G^2$ using 10 colors except the vertices incident to $f$, $f'$, and $f''$. So, we have $|L(v)|\ge 7$, $|L(v_1)|\ge 5$, $|L(v_2)|\ge 4$, $|L(v_3)|\ge 5$, $|L(v_4)|\ge 5$, $|L(v_5)|\ge 4$, $|L(v_6)|\ge 5$, $|L(v_7)|\ge 4$, and $|L(v_8)|\ge 5$. 
By \Cref{cc1}, $d(v_2,v_7)=3$. Since $|L(v_2)|+|L(v_7)|>5$, we can color $v_2$ by $c$ and $v_7$ by $c'$ such that $|L(v_8)\backslash \{c,c'\}|\ge 4$. Color $v_2$ by $c$ and $v_7$ by $c'$ and call the new list of available colors $L_1$.

So, we have $|L_1(v)|\ge 5$, $|L_1(v_1)|\ge 3$, $|L_1(v_3)|\ge 4$, $|L_1(v_4)|\ge 4$, $|L_1(v_5)|\ge 3$, $|L_1(v_6)|\ge 4$, and $|L_1(v_8)|\ge 4$. Since $G$ is planar with $g(G) \ge 4$, we deduce that $d(v_1,v_5)=3$.  
 Since $|L_1(v_1)|+|L_1(v_5)|>5$, we can color $v_1$ by $c_1$ and $v_5$ by $c_2$ such that $|L_1(v)\backslash \{c_1,c_2\}|\ge 4$. Color $v_1$ by $c$ and $v_5$ by $c'$ and call the new list of available colors $L_2$.

 So, we have $|L_2(v)|\ge 4$, $|L_2(v_3)|\ge 2$, $|L_2(v_4)|\ge 3$, $|L_2(v_6)|\ge 2$, and $|L_2(v_8)|\ge 3$. Since $g(G)\ge 4$, we deduce that $d(v_3,v_8)=3$.

    We will show that $L_2(v_6)\subset L_2(v)$. 
     Suppose that $L_2(v_6)\nsubseteq L_2(v)$. Color $v_6$ by $i\in L_2(v_6)\backslash L_2(v)$ and call the new list $L_3$. So, we have $|L_3(v)|\ge 4$, $|L_3(v_3)|\ge 1$, $|L_3(v_4)|\ge 2$, and $|L_3(v_8)|\ge 2$.
     Greedily color in order $v_3$, $v_4$, $v_8$, and $v$ in $G^2$ to obtain $\chi_2(G)\leq 10$, a contradiction. Hence, $L_2(v_6)\subset L_2(v)$.

     We will now show that $L_2(v_4)\subset L_2(v)$. 
     Suppose that $L_2(v_4)\nsubseteq L_2(v)$. 
     Color $v_4$ by $i\in L_2(v_4)\backslash L_2(v)$ and call the new list $L_3$. Since $L_2(v_6)\subset L_2(v)$, $i\notin L_2(v_6)$. So, we have $|L_3(v)|\ge 4$, $|L_3(v_3)|\ge 1$, $|L_3(v_6)|\ge 2$, and $|L_3(v_8)|\ge 2$.
     Greedily color in order $v_3$, $v_6$, $v_8$, and $v$ in $G^2$ to obtain $\chi_2(G)\leq 10$, a contradiction. Hence, $L_2(v_4)\subset L_2(v)$.
     
     Recall that $d(v_3,v_8)=3$. Therefore, since $|L_2(v_3)|+|L(v_8)|>4$ and $L_2(v_4)\cup L_2(v_6)\subset L(v)$, we can color $v_3$ by $c_3$ and $v_8$ by $c_4$ such that $|L_2(v)\backslash \{c_3,c_4\}|\ge 3$, $|L_2(v_4)\backslash \{c_3,c_4\}|\ge 2$, and $|L_2(v_6)\backslash \{c_3,c_4\}|\ge 1$. Color $v_3$ by $c_3$ and $v_8$ by $c_4$ then greedily color in order $v_6$, $v_4$, and $v$ in $G^2$ to obtain $\chi_2(G)\leq 10$, a contradiction.
\end{proof}
\begin{Claim} \label{cc3}
    $d(v_1,v_5)=3$.
\end{Claim}
\begin{proof}
    Suppose, to the contrary, that $d(v_1,v_5)\leq 2$. Suppose that $d(v_1,v_5)=1$. Then, $v_1$ is contained in two adjacent $4$-cycle $vv_1v_2v_3$ and $vv_1v_5v_6$, a contradiction by Lemma \ref{l9}. Thus, $d(v_1,v_5)=2$. By \Cref{l9}, we deduce that $x\neq v_2$. By \Cref{l8}, we deduce that $x\neq v_8$.
Suppose that $x=v_4$. Then, $v_4$ is contained in two $4$-cycles $v_1v_2v_3v_4$ and $vv_1v_4v_3$ and has a $3$-neighbor $v_5$, a contradiction by Lemma \ref{l10}. Thus, $x$ is not incident to $f$, $f'$, or $f''$.
Since $G$ is $4$-irregular, $d(x)=3$.

    Greedily color all the vertices in $G^2$ using 10 colors except the vertices incident to $f$, $f'$, and $f''$. So, we have $|L(v)|\ge 7$, $|L(v_1)|\ge 6$, $|L(v_2)|\ge 4$, $|L(v_3)|\ge 5$, $|L(v_4)|\ge 4$, $|L(v_5)|\ge 6$, $|L(v_6)|\ge 5$, $|L(v_7)|\ge 4$, and $|L(v_8)|\ge 4$. Since $G$ is planar with $g(G)\ge 4$, we deduce that $d(v_4,v_7)=3$. Hence, since $|L(v_4)|+|L(v_7)|>6$, we can color $v_4$ by $c$ and $v_7$ by $c'$ such that $|L(v_5)\backslash \{c,c'\}|\ge 5$. Color $v_4$ by $c$ and $v_7$ by $c'$ and call the new list of available colors $L_1$. 

So, we have $|L_1(v)|\ge 5$, $|L_1(v_1)|\ge 5$, $|L_1(v_2)|\ge 3$, $|L_1(v_3)|\ge 4$, $|L_1(v_5)|\ge 5$, $|L_1(v_6)|\ge 3$, and $|L_1(v_8)|\ge 3$.
    Since $G$ is $4$-irregular and planar with $g(G)\ge 4$, we deduce that $d(v_2,v_6)=3$. Since $|L_1(v_2)|+|L_1(v_6)|>5$, we can color $v_2$ by $c_1$ and $v_6$ by $c_2$ such that $|L_1(v)\backslash \{c_1,c_2\}|\ge 4$. Color $v_2$ by $c_1$ and $v_6$ by $c_2$ and call the new list of available colors $L_2$.

So, we have $|L_2(v)|\ge 4$, $|L_2(v_1)|\ge 3$,  $|L_2(v_3)|\ge 2$, $|L_2(v_5)|\ge 4$,  and $|L_2(v_8)|\ge 1$. Greedily color in order $v_8$, $v_3$, $v_1$, $v$, and $v_5$ to obtain $\chi_2(G)\leq 10$, a contradiction.
\end{proof}
 Now, greedily color all the vertices in $G^2$ using 10 colors except the vertices incident to $f$, $f'$, and $f''$. So, we have $|L(v)|\ge 7$, $|L(v_1)|\ge 5$, $|L(v_2)|\ge 4$, $|L(v_3)|\ge 5$, $|L(v_4)|\ge 4$, $|L(v_5)|\ge 4$, $|L(v_6)|\ge 5$, $|L(v_7)|\ge 4$, and $|L(v_8)|\ge 4$. By \Cref{cc2}, $d(v_4,v_8)\ge 3$. Since $|L(v_4)|+|L(v_8)|>7$, we can color $v_4$ by $c$ and $v_8$ by $c'$ such that $|L(v)\backslash \{c,c'\}|\ge 6$. Color $v_4$ by $c$ and $v_8$ by $c'$ and call the new list of available colors $L_1$. 

So, we have $|L_1(v)|\ge 6$, $|L_1(v_1)|\ge 4$, $|L_1(v_2)|\ge 2$, $|L_1(v_3)|\ge 4$, $|L_1(v_5)|\ge 3$, $|L_1(v_6)|\ge 3$, and $|L_1(v_7)|\ge 3$. By \Cref{cc1}, $d(v_2,v_7)= 3$. Since $|L_1(v_2)|+|L_1(v_7)|>4$, we can color $v_2$ by $c_1$ and $v_7$ by $c_2$ such that $|L_1(v_1)\backslash \{c_1,c_2\}|\ge 3$. Color $v_2$ by $c_1$ and $v_7$ by $c_2$ and call the new list of available colors $L_2$. 

So, we have $|L_2(v)|\ge 4$, $|L_2(v_1)|\ge 3$,  $|L_2(v_3)|\ge 3$, $|L_2(v_5)|\ge 2$, and $|L_2(v_6)|\ge 2$. 
\begin{Claim} \label{cc5}
    $L_2(v_6)\subset L_2(v)$.
\end{Claim}
\begin{proof}
    Suppose, to the contrary, that $L_2(v_6)\nsubseteq L_2(v)$. Color $v_6$ by $i\in L_2(v_6)\backslash L_2(v)$ and call the new list of available colors $L_3$. So, we have $|L_3(v)|\ge 4$, $|L_3(v_1)|\ge 2$, $|L_3(v_3)|\ge 2$, and $|L_3(v_5)|\ge 1$. Greedily color in order $v_5$, $v_3$, $v_1$, and $v$ in $G^2$ to obtain $\chi_2(G) \leq 10$, a contradiction.
\end{proof}

\begin{Claim} \label{cc6}
    $L_2(v_3)\subset L_2(v)$.
\end{Claim}
\begin{proof}
    Suppose, to the contrary, that $L_2(v_3)\nsubseteq L_2(v)$. Color $v_3$ by $i\in L_2(v_3)\backslash L_2(v)$ and call the new list of available colors $L_3$. Since $L_2(v_6)\subset L_2(v)$ by \Cref{cc5}, we deduce that $i\notin L_2(v_6)$. So, we have $|L_3(v)|\ge 4$, $|L_3(v_1)|\ge 2$, $|L_3(v_5)|\ge 1$, and $|L_3(v_6)|\ge 2$. Greedily color in order $v_5$, $v_6$, $v_1$, and $v$ in $G^2$ to obtain $\chi_2(G) \leq 10$, a contradiction.
\end{proof}
\textbf{Final Step:} By  \Cref{cc3}, we have $d(v_1,v_5)=3$. Since $|L_2(v_1)|+|L_2(v_5)|>4$, we can color $v_1$ by $c_3$ and $v_5$ by $c_4$ such that $|L_2(v)\backslash \{c_3,c_4\}|\ge 3$. Since $L_2(v_3)\cup L_2(v_6)\subset L_2(v)$ by \Cref{cc5} and \ref{cc6}, we deduce that $|L_2(v_3)\backslash \{c_3,c_4\}|\ge 2$ and $|L_2(v_6)\backslash \{c_3,c_4\}|\ge 1$. Color $v_1$ by $c_3$ and $v_5$ by $c_4$ and call the new list of available colors $L_3$. So, we have $|L_3(v)|\ge 3$, $|L_3(v_3)|\ge 2$, and $|L_3(v_6)|\ge 1$. Greedily color in order $v_6$, $v_3$, and $v$ in $G^2$ to obtain $\chi_2(G)\leq 10$, a contradiction.
\end{proof}
    
\begin{lemma} \label{l23}
    Let $v$ be a $3(1)$-vertex incident to a $6$-face $f$ Then, $f$ is incident to at most four vertices of degree 3.
\end{lemma}
\begin{proof}
    Let $v$ be a $3(1)$-vertex incident to a $6$-face $f$ and to a $4$-face $f'$. Suppose, to the contrary, that $f$ is incident to at least five vertices of degree $3$. Set $f'=vv_1v_2v_3$ and $f=vv_3v_4v_5v_6v_7$. Then, $d(v_1)=d(v_3)=4$ and   $d(v_2)=d(v_4)=d(v_5)=d(v_6)=d(v_7)=3$.
    We will study two cases depending on $d(v_2,v_6)$.
    \begin{itemize}
        \item \textbf{Case 1:} $d(v_2,v_6)\leq 2$. 
        \\
        By Lemma \ref{l11}, we deduce that $d(v_2,v_6)=2$. 
        Let $x\in N(v_2)\cap N(v_6)$. By Lemma \ref{l9}, we deduce that $x\notin \{v_1,v_3,v_5\}$. By \Cref{l11}, we deduce that $x\neq v_7$. Hence, since $g(G)\ge 4$, we deduce that $x$ is not incident to $f$ or $f'$.
Greedily color all the vertices in $G^2$ except the vertices incident to $f$ and $f'$. So, we have $|L(v)|\ge 6$, $|L(v_1)|\ge 3$,  $|L(v_2)|\ge 4$, $|L(v_3)|\ge 5$, $|L(v_4)|\ge 4$, $|L(v_5)|\ge 4$, $|L(v_6)|\ge 5$, and  $|L(v_7)|\ge 5$. Since $g(G)\ge 4$, we deduce that $d(v_2,v_7)=3$. Since $|L(v_2)|+|L(v_7)|>6$, we can color $v_2$ by $c$ and $v_7$ by $c'$ such that $|L(v)\backslash \{c,c'\}|\ge 5$. Color $v_2$ by $c$ and $v_7$ by $c'$ and call the new list of available colors $L_1$. 

So, we have $|L_1(v)|\ge 5$,  $|L_1(v_1)|\ge 1$, $|L_1(v_3)|\ge 3$, $|L_1(v_4)|\ge 3$, $|L_1(v_5)|\ge 3$, and  $|L_1(v_6)|\ge 3$. Greedily color in order $v_1$, $v_3$, $v_4$, $v_5$, $v_6$, and $v$ in $G^2$ to obtain $\chi_2(G)\leq 10 $, a contradiction.
        \item \textbf{Case 2:} $d(v_2,v_6)\ge 3$. \\
        Greedily color all the vertices in $G^2$ except the vertices incident to $f$ and $f'$. So, we have $|L(v)|\ge 6$, $|L(v_1)|\ge 3$,  $|L(v_2)|\ge 3$, $|L(v_3)|\ge 5$, $|L(v_4)|\ge 4$, $|L(v_5)|\ge 4$, $|L(v_6)|\ge 4$, and  $|L(v_7)|\ge 5$. Since $d(v_2,v_6)=3$ and $|L(v_2)|+|L(v_6)|>6$, we can color $v_2$ by $c$ and $v_7$ by $c'$ such that $|L(v)\backslash \{c,c'\}|\ge 5$. Color $v_2$ by $c$ and $v_6$ by $c'$ and call the new list of available colors $L_1$. 
        
        So, we have $|L_1(v)|\ge 5$,  $|L_1(v_1)|\ge 2$, $|L_1(v_3)|\ge 4$, $|L_1(v_4)|\ge 2$, $|L_1(v_5)|\ge 3$, and  $|L_1(v_7)|\ge 4$. Greedily color in order $v_1$, $v_4$, $v_5$, $v_3$, $v_7$, and $v$ in $G^2$ to obtain $\chi_2(G)\leq 10 $, a contradiction. \qedhere
    \end{itemize}
\end{proof}
\begin{lemma} \label{l24}
    Let $v$ be a weak $3(1)$-vertex incident to a $6$-face $f$ Then, $f$ is incident to at most three vertices of degree 3.
\end{lemma}
\begin{proof}
    Let $v$ be a weak $3(1)$-vertex incident to a $6$-face $f$ and to a $4$-face $f'$. Suppose, to the contrary, that $f$ is incident to at at least four vertices of degree $3$. By \Cref{l23}, we deduce that $f$ is incident to exactly four vertices of degree $3$. Set $f'=vv_1v_2v_3$ and $f=vv_3v_4v_5v_6v_7$. Let $f''=vv_7v_8v_9v_1$ be the $5$-face incident to $v$. Then,  $d(v_1)=d(v_3)=4$ and $d(v_2)=d(v_4)=d(v_9)=3$. By Lemma \ref{l19}, we deduce that $f''$ is a good face. We will study two cases depending on $d(v_7)$. 
    \begin{itemize}
        \item \textbf{Case 1:} $d(v_7)=4$.\\
        Then, $d(v_4)=d(v_5)=d(v_6)=d(v_8)=3$.
          We will study two cases according to $d(v_2,v_6)$.
     \begin{itemize}
     \item \textbf{Subcase 1.1:} $d(v_2,v_6)\leq 2$. \\ 
     Greedily color all the vertices in $G^2$ using $10$ colors except the vertices incident to $f$, $f'$, and $f''$. So, we have $|L(v)|\ge 7$, $|L(v_1)|\ge 5$,  $|L(v_2)|\ge 5$, $|L(v_3)|\ge 5$, $|L(v_4)|\ge 4$, $|L(v_5)|\ge 4$, $|L(v_6)|\ge 5$, $|L(v_7)|\ge 5$, $|L(v_8)|\ge 4$, and $|L(v_9)|\ge 4$. Since $|L(v_7)|>4$, we can color $v_7$ by $c$ such that $|L(v_5)\backslash \{c\}|\ge 4$. Color $v_7$ by $c$ and call the new list of available colors $L_1$.

     So, we have $|L_1(v)|\ge 6$, $|L_1(v_1)|\ge 4$,  $|L_1(v_2)|\ge 5$, $|L_1(v_3)|\ge 4$, $|L_1(v_4)|\ge 4$, $|L_1(v_5)|\ge 4$, $|L_1(v_6)|\ge 4$, $|L_1(v_8)|\ge 3$, and $|L_1(v_9)|\ge 3$.
Since $|L_1(v_2)|>4$, we can color $v_2$ by $c'$ such that $|L_1(v_6)\backslash \{c'\}|\ge 4$. Color $v_2$ by $c'$ and call the new list of available colors $L_2$.

 So, we have $|L_2(v)|\ge 5$, $|L_2(v_1)|\ge 3$, $|L_2(v_3)|\ge 3$, $|L_2(v_4)|\ge 3$, $|L_2(v_5)|\ge 4$, $|L_2(v_6)|\ge 4$, $|L_2(v_8)|\ge 3$ and $|L_2(v_9)|\ge 2$. Greedily color in order $v_9$, $v_8$, $v_1$, $v_3$, $v$, $v_4$, $v_6$, and $v_5$ in $G^2$ to obtain $\chi_2(G)\leq 10$, a contradiction.
 \item \textbf{Subcase 1.2:} 
$d(v_2,v_6)=3$.
 \\
 Greedily color all the vertices in $G^2$ using $10$ colors except the vertices incident to $f$, $f'$, and $f''$. So, we have $|L(v)|\ge 7$, $|L(v_1)|\ge 5$,  $|L(v_2)|\ge 4$, $|L(v_3)|\ge 5$, $|L(v_4)|\ge 4$, $|L(v_5)|\ge 4$, $|L(v_6)|\ge 4$, $|L(v_7)|\ge 5$, $|L(v_8)|\ge 4$, and $|L(v_9)|\ge 4$.  Since $|L(v_7)|>4$, we can color $v_7$ by $c$ such that $|L(v_5)\backslash \{c\}|\ge 4$. Color $v_7$ by $c$ and call the new list of available colors $L_1$.

     So, we have $|L_1(v)|\ge 6$, $|L_1(v_1)|\ge 4$,  $|L_1(v_2)|\ge 4$, $|L_1(v_3)|\ge 4$, $|L_1(v_4)|\ge 4$, $|L_1(v_5)|\ge 4$, $|L_1(v_6)|\ge 3$, $|L_1(v_8)|\ge 3$, and $|L_1(v_9)|\ge 3$. Since $d(v_2,v_6)\ge 3$ and $|L_1(v_2)|+|L_1(v_6)|>4$, we can color $v_2$ by $c_1$ and $v_6$ by $c_2$ such that $|L_1(v_4)\backslash \{c_1,c_2\}|\ge 3$. Color $v_2$ by $c_1$ and $v_6$ by $c_2$ and call the new list of available colors $L_2$.

     So, we have $|L_2(v)|\ge 4$, $|L_2(v_1)|\ge 3$, $|L_2(v_3)|\ge 3$, $|L_2(v_4)|\ge 3$, $|L_2(v_5)|\ge 3$, $|L_2(v_8)|\ge 2$ and $|L_2(v_9)|\ge 2$. Greedily color in order $v_9$, $v_8$, $v_1$, $v$, $v_3$, $v_4$, and $v_5$ in $G^2$ to obtain $\chi_2(G)\leq 10$, a contradiction.
     \end{itemize}
     \item \textbf{Case 2:} $d(v_7)=3$. \\
     Then, $d(v_8)=3$ since $f''$ is a good $5$-face.
  Since $f$ is incident to four vertices of degree $3$,  either $d(v_5)=3$ and $d(v_6)=4$ or $d(v_5)=4$ and $d(v_6)=3$.
     \begin{itemize}
         \item \textbf{Subcase 2.1:} $d(v_5)=4$ and $d(v_6)=3$.
\\
We will show that $d(v_1,v_6)=3$. Suppose, to the contrary, that $d(v_1,v_4)\leq 2$. Greedily color all the vertices in $G^2$ using $10$ colors except the vertices incident to $f$, $f'$, and $f''$.
So, we have $|L(v)|\ge 8$, $|L(v_1)|\ge 6$,  $|L(v_2)|\ge 4$, $|L(v_3)|\ge 5$, $|L(v_4)|\ge 3$, $|L(v_5)|\ge 2$, $|L(v_6)|\ge 5$, $|L(v_7)|\ge 7$, $|L(v_8)|\ge 3$, and $|L(v_9)|\ge 3$. By \Cref{l8}, we deduce that $v_9\notin N(v_6)$. By Lemma \ref{l9}, we deduce that $v_2\notin  N(v_5)$. By \Cref{l8} and since $G$ is planar with $g(G)\ge 4$, we deduce that $d(v_1,v_4)=3$. Thus, since $|L(v_1)|+|L(v_4)|>8$, we can color $v_1$ by $c$ and $v_4$ by $c'$ such that $|L(v)\backslash \{c,c'\}|\ge 7$. Color $v_1$ by $c$ and $v_4$ by $c'$ and call the new list of available colors $L_1$. So, we have $|L_1(v)|\ge 7$, $|L_1(v_2)|\ge 2$,  $|L_1(v_3)|\ge 3$, $|L_1(v_5)|\ge 1$, $|L_1(v_6)|\ge 3$, $|L_1(v_7)|\ge 6$, $|L_1(v_8)|\ge 2$, and $|L_1(v_9)|\ge 2$. Greedily color in order $v_5$, $v_2$, $v_9$, $v_8$, $v_3$, $v_6$, $v_7$, and $v$ in $G^2$ to obtain $\chi_2(G)\leq 10$, a contradiction. Hence, $d(v_1,v_6)\ge 3$.

     Greedily color all the vertices in $G^2$ except the vertices incident to $f$, $f'$, and $f''$. So, we have $|L(v)|\ge 8$, $|L(v_1)|\ge 5$,  $|L(v_2)|\ge 4$, $|L(v_3)|\ge 5$, $|L(v_4)|\ge 3$, $|L(v_5)|\ge 2$, $|L(v_6)|\ge 4$, $|L(v_7)|\ge 7$, $|L(v_8)|\ge 3$, and $|L(v_9)|\ge 3$. Since $d(v_1,v_6)=3$ and  $|L(v_3)|+|L(v_8)|>8$, we can color $v_1$ by $c$ and $v_6$ by $c'$ such that $|L(v)\backslash \{c,c'\}|\ge 7$. Color $v_1$ by $c$ and $v_6$ by $c'$ and call the new list of available colors $L_1$.

     So, we have $|L_1(v)|\ge 7$, $|L_1(v_2)|\ge 3$,  $|L_1(v_3)|\ge 4$, $|L_1(v_4)|\ge 2$, $|L_1(v_5)|\ge 1$, $|L_1(v_7)|\ge 5$, $|L_1(v_8)|\ge 1$, and $|L_1(v_9)|\ge 2$. Greedily color in order $v_5$, $v_8$, $v_9$, $v_4$, $v_2$, $v_3$, $v_7$, and $v$ in $G^2$ to obtain $\chi_2(G)\leq 10$, a contradiction.
     \item \textbf{Subcase 2.2:} $d(v_5)=3$ and $d(v_6)=4$.
\\
 We will show that $d(v_1,v_4)=3$. Suppose, to the contrary, that $d(v_1,v_4)\leq 2$. If $d(v_1,v_4)=1$, then $v$ is contained in two $4$-cycles and has a $3$-neighbor $v_7$, a contradiction by Lemma \ref{l10}. Thus, $d(v_1,v_4)=2$. Let $x\in N(v_1)\cap N(v_4)$. Since $g(G)\ge 4$, $x \notin \{v_2,v_3\}$. We will show that $x\neq v_5$. Suppose that $x=v_5$. 
 Greedily color all the vertices in $G^2$ using $10$ colors except the vertices incident to $f$, $f'$, and $f''$. So, we have $|L(v)|\ge 9$, $|L(v_1)|\ge 8$,  $|L(v_2)|\ge 5$, $|L(v_3)|\ge 5$, $|L(v_4)|\ge 5$, $|L(v_5)|\ge 7$, $|L(v_6)|\ge 4$, $|L(v_7)|\ge 6$, $|L(v_8)|\ge 3$, and $|L(v_9)|\ge 4$.  Since $|L(v_7)|>5$ and $|L(v_1)|>6$, we can color $v_7$ by $c$ and $v_1$ by $c'$ such that $|L(v_3)\backslash \{c,c'\}|\ge 5$.  Color $v_7$ by $c$ and $v_1$ by $c'$ and call the new list of available colors $L_1$. So, we have $|L_1(v)|\ge 7$, $|L_1(v_2)|\ge 4$,  $|L_1(v_3)|\ge 5$, $|L_1(v_4)|\ge 4$, $|L_1(v_5)|\ge 5$, $|L_1(v_6)|\ge 2$, $|L_1(v_8)|\ge 1$, and $|L_1(v_9)|\ge 2$. Greedily color in order $v_8$, $v_9$, $v_6$, $v_4$, $v_5$, $v_2$, $v$, and $v_3$ in $G^2$ to obtain $\chi_2(G)\leq 10$, a contradiction. Thus, $x\neq v_5$. Therefore, $v_4$ has two neighbors of degree $3$. Moreover, since $G$ is planar, we deduce that $d(v_2,v_7)=3$.
 Greedily color all the vertices in $G^2$ using $10$ colors except the vertices incident to $f$, $f'$, and $f''$. So, we have $|L(v)|\ge 8$, $|L(v_1)|\ge 6$,  $|L(v_2)|\ge 4$, $|L(v_3)|\ge 5$, $|L(v_4)|\ge 6$, $|L(v_5)|\ge 3$, $|L(v_6)|\ge 3$, $|L(v_7)|\ge 6$, $|L(v_8)|\ge 3$, and $|L(v_9)|\ge 3$.
 Since $d(v_2,v_7)=3$ and $|L(v_2)|+|L(v_7)|>8$, we can color $v_2$ by $c$ and $v_7$ by $c'$ such that $|L(v)\backslash \{c,c'\}|\ge 7$. Color $v_2$ by $c$ and $v_7$ by $c'$ and call the new list of available colors $L_1$. 
 So, we have $|L_1(v)|\ge 7$, $|L_1(v_1)|\ge 4$,  $|L_1(v_3)|\ge 3$, $|L_1(v_4)|\ge 5$, $|L_1(v_5)|\ge 2$, $|L_1(v_6)|\ge 2$, $|L_1(v_8)|\ge 2$, and $|L_1(v_9)|\ge 1$. Greedily color in order $v_9$, $v_8$, $v_6$, $v_5$, $v_1$, $v_3$, $v_4$, and $v$ in $G^2$ to obtain $\chi_2(G)\leq 10$, a contradiction. Therefore, $d(v_1,v_4)=3$.

     Now, greedily color all the vertices in $G^2$ using $10$ colors except the vertices incident to $f$, $f'$, and $f''$. So, we have $|L(v)|\ge 8$, $|L(v_1)|\ge 5$,  $|L(v_2)|\ge 4$, $|L(v_3)|\ge 5$, $|L(v_4)|\ge 4$, $|L(v_5)|\ge 3$, $|L(v_6)|\ge 3$, $|L(v_7)|\ge 6$, $|L(v_8)|\ge 3$, and $|L(v_9)|\ge 3$.  
     
     We will show that $L(v_2)\subset L(v)$. Suppose, to the contrary, that $L(v_2)\nsubseteq L(v)$. Color $v_2$ by $i\in L(v_2)\backslash L(v)$ and call the new list of available colors $L_1$. So, we have $|L_1(v)|\ge 8$, $|L_1(v_1)|\ge 4$,  $|L_1(v_3)|\ge 4$, $|L_1(v_4)|\ge 3$, $|L_1(v_5)|\ge 3$, $|L_1(v_6)|\ge 3$ $|L_1(v_7)|\ge 6$, $|L_1(v_8)|\ge 3$, and $|L_1(v_9)|\ge 2$. 
     Since $d(v_1,v_4)=3$ and $|L_1(v_1)|+|L_1(v_4)|>4$, we can color $v_1$ by $c_1$ and $v_4$ by $c_2$ such that $|L_1(v_3)\backslash \{c_1,c_2\}|\ge 3$. Color $v_1$ by $c_1$ and $v_4$ by $c_2$ and call the new list of available colors $L_2$. So, we have $|L_2(v)|\ge 6$, $|L_2(v_3)|\ge 3$, $|L_2(v_5)|\ge 2$, $|L_2(v_6)|\ge 2$ $|L_2(v_7)|\ge 5$, $|L_1(v_8)|\ge 2$, and $|L_2(v_9)|\ge 1$. 
     Greedily color in order $v_9$, $v_8$, $v_6$, $v_5$, $v_7$, $v_3$, and $v$ in $G^2$ to obtain $\chi_2(G)\leq 10$, a contradiction. Thus, $L(v_2)\subset L(v)$. 

     Since $d(v_1,v_4)=3$ and $|L(v_1)|+|L(v_4)|>8$, we can color $v_1$ by $c$ and $v_4$ by $c'$ such that $|L(v)\backslash \{c,c'\}|\ge 7$.  Since $L(v_2)\subset L(v)$, we deduce that $|L(v_2)\backslash \{c,c'\}|\ge 3$. Color $v_1$ by $c$ and $v_6$ by $c'$ and call the new list of available colors $L_1$.

     So, we have $|L_1(v)|\ge 7$, $|L_1(v_2)|\ge 3$,  $|L_1(v_3)|\ge 3$, $|L_1(v_5)|\ge 2$, $|L_1(v_6)|\ge 2$, $|L_1(v_7)|\ge 5$, $|L_1(v_8)|\ge 2$, and $|L_1(v_9)|\ge 2$. Greedily color in order $v_8$, $v_9$, $v_6$, $v_5$, $v_7$, $v_3$, $v_2$, and $v$ in $G^2$ to obtain $\chi_2(G)\leq 10$, a contradiction. \qedhere  \end{itemize}
    \end{itemize}
\end{proof}
\begin{lemma} \label{l25}
    Let $f=vv_1v_2v_3v_4v_5$ be a $6$-face such that $d(v)=2$ and $d(v_2)=d(v_3)=d(v_4)=3$. Then, $v_2$, $v_3$, and $v_4$ are good $3$-vertices.
\end{lemma}
\begin{proof}
   By Lemma \ref{l7} and \ref{l8}, we deduce that $v_2$, $v_3$, and $v_4$ are not $3(1)$-vertices.
   We will show that the third neighbors of $v_2$, $v_3$, and $v_4$ are of degree $4$.

   Suppose that the third neighbor of $v_2$ is of degree $3$. Greedily color all the vertices in $G^2$ using $10$ colors except $v$, $v_1$ and $v_2$. Then, we have $|L(v)|\ge 4$, $|L(v_1)|\ge 1$, and $|L(v_2)|\ge 2$. Greedily color in order $v_1$, $v_2$, and $v$ in $G^2$ to obtain $\chi_2(G)\leq 10$, a contradiction. Thus, the third neighbor of $v_2$ is of degree $4$. Similarly, the third neighbor of $v_4$ is of degree $4$.

   Now, we will show that the third neighbor of $v_3$ is of degree $4$. Suppose, to the contrary, that the third neighbor of $v_3$ is of degree $3$. Greedily color all the vertices in $G^2$ using $10$ colors except $v$, $v_1$, $v_2$, $v_3$. Then, we have $|L(v)|\ge 4$, $|L(v_1)|\ge 2$, $|L(v_2)|\ge 2$, $|L(v_3)|\ge 3$. Greedily color in order $v_1$, $v_2$, $v_3$, and $v$ in $G^2$ to obtain $\chi_2(G)\leq 10$, a contradiction. Thus, the third neighbor of $v_3$ is of degree $4$.

   Since the third neighbors of $v_2$, $v_3$, and $v_4$ are of degree $4$, we deduce that they are good $3$-vertices.
\end{proof}

\begin{figure}[htbp]
    \centering
    \begin{tikzpicture}[
        scale=1,
        vertex/.style={circle, fill=black, inner sep=2pt},
        line/.style={thick}
    ]
        \node[vertex, label=below left:$v_6$]  (v6) at (0, 0) {};
        \node[vertex, label=below:$v_7$]       (v7) at (1, 0) {};
        \node[vertex, label=below:$x$]         (x)  at (2, 0) {};
        \node[vertex, label=below right:$y$]   (y)  at (4, 0) {};

        \node[vertex, label=left:$v_5$]        (v5) at (0, 2) {};
        \node[vertex, label=above right:$v_4$] (v4) at (2, 2) {};
        \node[vertex, label=right:$v_3$]       (v3) at (4, 2) {};

        \node[vertex, label=above left:$v$]    (v)  at (0, 4) {};
        \node[vertex, label=above:$v_1$]       (v1) at (2, 4) {};
        \node[vertex, label=above right:$v_2$] (v2) at (4, 4) {};

        \draw[line] (v) -- (v1) -- (v2);
        \draw[line] (v5) -- (v4) -- (v3);
        \draw[line] (v6) -- (v7) -- (x) -- (y);

        \draw[line] (v) -- (v5) -- (v6);
        \draw[line] (v4) -- (x);
        \draw[line] (v2) -- (v3) -- (y);
    \end{tikzpicture}
    \caption{Illustration of \Cref{l26}}
    \label{f4}
\end{figure}

\begin{lemma} \label{l26}
    Let $f=vv_1v_2v_3v_4v_5$ be a $6$-face such that $d(v)=2$ and $d(v_2)=d(v_4)=3$. Suppose that $v_2$ and $v_4$ are $3(1)$-vertices. Then, $v_2$ and are $v_4$ are strong $3(1)$-vertices.
\end{lemma}
\begin{proof}
    We will show that $v_4$ is a strong $3(1)$-vertex. Similarly, we can show that $v_2$ is a strong $3(1)$-vertex.
    Suppose that $v_4$ is a weak $3(1)$-vertex. By Lemma \ref{l7}, we deduce that the $4$-face $f'$ incident to $v_4$ is of the form $v_3v_4xy$. Let $f''=v_4v_5v_6v_7x$ be the $5$-face incident to $v_4$ (See \Cref{f4}). So, we have $d(v_3)=d(x)=4$ and $d(y)=d(v_6)=d(v_7)=3$.
 Greedily color all the vertices in $G^2$ using $10$ colors except the vertices incident to $f$, $f'$, or $f''$. Then, we have $|L(v)|\ge 7$, $|L(v_1)|\ge 3$,  $|L(v_2)|\ge 3$, $|L(v_3)|\ge 5$, $|L(v_4)|\ge 7$, $|L(v_5)|\ge 6$, $|L(v_6)|\ge 4 $, $|L(v_7)|\ge 4$, $|L(x)|\ge 5$, and $|L(y)|\ge 4$. Since $|L(x)|>4$, we can color $x$ by $c$ such that $|L(v_6)\backslash \{c\}|\ge 4$. Color $x$ by $c$ and call the new list of available colors $L_1$.

    Then, we have $|L_1(v)|\ge 7$, $|L_1(v_1)|\ge 3$,  $|L_1(v_2)|\ge 3$, $|L_1(v_3)|\ge 4$, $|L_1(v_4)|\ge 6$, $|L_1(v_5)|\ge 5$, $|L_1(v_6)|\ge 4 $, $|L_1(v_7)|\ge 3$, and $|L_1(y)|\ge 3$. Greedily color in order $v_7$, $y$, $v_2$, $v_3$, $v_1$, $v_4$, $v_5$, $v_6$, and $v$ in $G^2$ to obtain $\chi_2(G)\leq 10$, a contradiction.
    \end{proof}
\begin{figure}[htbp]
    \centering
    \begin{subfigure}[b]{0.48\textwidth}
        \centering
        \begin{tikzpicture}[
            scale=1.0,
            vertex/.style={circle, fill=black, inner sep=2pt},
            line/.style={thick}
        ]
            \node[vertex, label=below left:$v_6$]  (v6) at (0, 0) {};
            \node[vertex, label=below:$v_5$]       (v5) at (2, 0) {};
            \node[vertex, label=below right:$v_4$] (v4) at (4, 0) {};

            \node[vertex, label=below right:$v_1$] (v1) at (0, 2) {};
            \node[vertex, label=below right:$v_2$] (v2) at (2, 2) {};
            \node[vertex, label=below right:$v_3$] (v3) at (4, 2) {};

            \node[vertex, label=above left:$v_9$]  (v9) at (0, 3.5) {};
            \node[vertex, label=above:$v_8$]       (v8) at (1, 3.5) {};
            \node[vertex, label=above right:$v_7$] (v7) at (2, 3.5) {};

            \draw[line] (v6) -- (v5) -- (v4);
            \draw[line] (v1) -- (v2) -- (v3);
            \draw[line] (v9) -- (v8) -- (v7);

            \draw[line] (v9) -- (v1) -- (v6);
            \draw[line] (v7) -- (v2);
            \draw[line] (v3) -- (v4);
        \end{tikzpicture}
        
        \label{fig:graph_3238}
    \end{subfigure}
    \hfill
    \begin{subfigure}[b]{0.48\textwidth}
        \centering
        \begin{tikzpicture}[
            scale=1.0,
            vertex/.style={circle, fill=black, inner sep=2pt},
            line/.style={thick}
        ]
            \node[vertex, label=below left:$v_2$]  (v2) at (0, 0) {};
            \node[vertex, label=below:$v_3$]       (v3) at (2, 0) {};
            \node[vertex, label=below right:$v_4$] (v4) at (4, 0) {};

            \node[vertex, label=below right:$v_1$] (v1) at (0, 2) {};
            \node[vertex, label=below right:$v_6$] (v6) at (2, 2) {};
            \node[vertex, label=below right:$v_5$] (v5) at (4, 2) {};

            \node[vertex, label=above left:$v_9$]  (v9) at (0, 3.5) {};
            \node[vertex, label=above:$v_8$]       (v8) at (1, 3.5) {};
            \node[vertex, label=above right:$v_7$] (v7) at (2, 3.5) {};

            \draw[line] (v2) -- (v3) -- (v4);
            \draw[line] (v1) -- (v6) -- (v5);
            \draw[line] (v9) -- (v8) -- (v7);

            \draw[line] (v9) -- (v1) -- (v2);
            \draw[line] (v7) -- (v6);
            \draw[line] (v5) -- (v4);
        \end{tikzpicture}
        
        \label{fig:graph_3239}
    \end{subfigure}
    \caption{Illustration of \Cref{l27}}
    \label{f5}
\end{figure}

    \begin{lemma}
        \label{l27}
        Let $f=v_1v_2v_3v_4v_5v_6$ be a $6$-face such that $f$ is incident to five $3$-vertices. Then, all incident $3$-vertices are good $3$-vertices.
    \end{lemma}
\begin{proof}
    Without loss of generality, suppose that $d(v_6)=4$. By \Cref{l23}, we deduce that $f$ is not incident to any $3(1)$-vertex. By \Cref{l11}, we deduce that the third neighbors of $v_2$, $v_3$, and $v_4$ are of degree $4$. Hence, $v_3$ is a good $3$-vertex. We will show that $v_1$ and $v_5$ are not incident to any bad or semi-bad $5$-face. Therefore, we deduce that $v_2$ and $v_4$ are not incident to any semi-bad or bad $5$-face.

    We will show that $v_1$ is not incident to any semi-bad or bad $5$-face. Similarly, we can  show that $v_5$ is not incident to any semi-bad or bad $5$-face.
Suppose, to the contrary, that $v_1$ is incident to a bad or semi-bad $5$-face $f'$ (See \Cref{f5}).
    \begin{itemize}
        \item \textbf{Case 1:} $f'=v_1v_2v_7v_8v_9$.
        \\
        Then, $d(v_7)=4$ since the third neighbor of $v_2$ is of degree $4$. So, $d(v_9)=d(v_8)=3$.
        Greedily color all the vertices in $G^2$ using $10$ colors except the vertices incident to $f$ and $f'$. So, we have $|L(v_1)|\ge 7$,  $|L(v_2)|\ge 7$, $|L(v_3)|\ge 5$, $|L(v_4)|\ge 4$, $|L(v_5)|\ge 3$, $|L(v_6)|\ge 3$, $|L(v_7)|\ge 3$, $|L(v_8)|\ge 3$, and $|L(v_9)|\ge 5$.  Since $|L(v_2)|>5$, we can color $v_2$ by $c$  such that $|L(v_3)\backslash \{c\}|\ge 5$.  Color $v_2$ by $c$ and call the new list of available colors $L_1$. 
        
        So, we have $|L_1(v_1)|\ge 6$, $|L_1(v_3)|\ge 5$, $|L_1(v_4)|\ge 3$, $|L_1(v_5)|\ge 3$, $|L_1(v_6)|\ge 2$, $|L_1(v_7)|\ge 2$, $|L_1(v_8)|\ge 2$, and $|L_1(v_9)|\ge 4$. Greedily color in order $v_7$, $v_8$, $v_6$, $v_9$, $v_5$, $v_1$, $v_4$, and $v_3$ in $G^2$ to obtain $\chi_2(G)\leq 10$, a contradiction.
        \item \textbf{Case 2:} $f'=v_1v_6v_7v_8v_9$. \\
        Then, $d(v_7)=d(v_8)=d(v_9)=3$. Greedily color all the vertices in $G^2$ using $10$ colors except the vertices incident to $f$ and $f'$. So, we have $|L(v_1)|\ge 7$,  $|L(v_2)|\ge 5$, $|L(v_3)|\ge 4$, $|L(v_4)|\ge 4$, $|L(v_5)|\ge 4$, $|L(v_6)|\ge 5$, $|L(v_7)|\ge 4$, $|L(v_8)|\ge 4$, and $|L(v_9)|\ge 5$. Since $|L(v_1)|>5$, we can color $v_1$ by $c$ such that $|L(v_2)\backslash \{c\}|\ge 5$.  Color $v_1$ by $c$ and call the new list of available colors $L_1$. 
        
        So, we have $|L_1(v_2)|\ge 5$, $|L_1(v_3)|\ge 3$, $|L_1(v_4)|\ge 4$, $|L_1(v_5)|\ge 3$, $|L_1(v_6)|\ge 4$, $|L_1(v_7)|\ge 3$, $|L_1(v_8)|\ge 3$, and $|L_1(v_9)|\ge 4$. Greedily color in order $v_7$, $v_8$, $v_6$, $v_9$, $v_5$, $v_4$, $v_3$, and $v_2$ in $G^2$ to obtain $\chi_2(G)\leq 10$, a contradiction.
    \end{itemize}
    Therefore, all incident $3$-vertices are good $3$-vertices.
\end{proof}
\begin{lemma}
    \label{l28}
    Let $f=v_1v_2v_3v_4v_5v_6v_7v_8$ be an $8$-face such that $f$ is incident to two vertices of degree $2$. Then, $f$ is incident to exactly two $3$-vertices and such vertices are good $3$-vertices.
    \begin{proof}
        By \Cref{l4}, the distance between the $2$-vertices is at least $4$. Without loss of generality, suppose that $d(v_1)=d(v_5)=2$. Then, $d(v_2)=d(v_4)=d(v_6)=d(v_8)=4$ by \Cref{l2}. Therefore, $d(v_3)=d(v_7)=3$ since $G$ is $4$-irregular. So, $f$ is incident to exactly two vertices of degree $3$. We will show that the third neighbor of $v_3$ is of degree $4$. Hence, since $v_3$ has three neighbors of degree $4$, we deduce that $v_3$ is a good $3$-vertex. Similarly, we can show that $v_7$ is a good $3$-vertex.

        Suppose, to the contrary, that $v_3$ has a $3$-neighbor. Greedily color all the vertices in $G^2$ using $10$ colors except $v_1$, $v_2$, $v_3$, $v_4$, and $v_5$. Then, we have $|L(v_1)|\ge 4$,  $|L(v_2)|\ge 2$, $|L(v_3)|\ge 3$, $|L(v_4)|\ge 2$, and $|L(v_5)|\ge 4$. Greedily color in order $v_2$, $v_4$, $v_3$, $v_1$, and $v_5$ in $G^2$ to obtain $\chi_2(G)\leq 10$, a contradiction. Thus, all neighbors of $v_3$ are of degree $4$. So, $v_3$ is a good $3$-vertex.
    \end{proof}
\end{lemma}
\subsection{Discharging}
Now, we will apply discharging method to prove that $G$ does not exist.\\
By Euler's formula, we have the following equality:
\\
$$\sum_{v \in V(G)}(d(v)-4)+ \sum_{f \in F(G) }(d(f)-4) = -8$$
We assign an initial charge $d(v)-4$ to every vertex $v$  and $d(f)-4$ to every face $f$, and
design appropriate discharging rules and redistribute charges among vertices and faces, such that the final charge of vertices
and faces are nonnegative, a contradiction. Hence, $G$ does not exist and \Cref{thm1} holds.
\\
Now, we design the following discharging rules:\\
\textbf{R1:} Every good $2$-vertex receives $1$ from each incident $6^+$-face.\\
\textbf{R2:} Every bad $2$-vertex receives $\frac{5}{3}$ from each incident $8^+$-face.\\
\textbf{R3:} Every good $5$-face sends $\frac{1}{3}$ to each  incident $3^-$-vertex.\\
\textbf{R4:} Every semi-bad $5$-face sends $\frac{1}{4}$ to each  incident 3-vertex. \\
\textbf{R5:} Every bad $5$-face sends $\frac{1}{5}$ to each  incident 3-vertex.  \\
\textbf{R6:} Every weak $3(1)$-vertex receives $\frac{2}{3}$ from each incident $6^+$-face.\\
\textbf{R7:} Every strong $3(1)$-vertex receives $\frac{1}{2}$ from each incident $6^+$-face.\\
\textbf{R8:} Every $6^+$-face sends $\frac{2}{5}$ to each incident bad $3$-vertex. \\
\textbf{R9:} Every $6^+$-face sends $\frac{5}{12}$ to each incident semi-bad $3$-vertex.\\
\textbf{R10:} Every $6^+$-face sends $\frac{1}{3}$ to each incident good $3$-vertex.\\

Note that weak $3(1)$-vertices are not adjacent by \Cref{l12}. Moreover, by \Cref{l8}, we deduce that a $3(1)$-vertex has at most one $3$-neighbor. Thus, each $k$-face is incident to at most $\lfloor \frac{2k}{3}\rfloor$ $3(1)$-vertices, for $k\ge 5$.
\\
\\

Denote by $\mu $ the final charge. \\ \\
\textbf{Final Charges:}
Let $f\in F(G)$ and $v\in V(G)$.
\begin{itemize}
    \item If $f$ is a $4$-face: $f$ does not send or receive any charge. So, $\mu(f)=0$.
    \item If $f$ is a $5$-face: Note that if $f$ is incident to $2$-vertex, we deduce by Lemma \ref{l2} that $f$ is a good $5$-face. Suppose that $f$ is a good face. Then, $f$ sends $\frac{1}{3}$ to three incident $3^-$-vertices. So, $\mu(f)= 0$. Suppose that $f$ is a semi-bad face. Then, $f$ sends $\frac{1}{4}$ to four incident $3$-vertices. So, $\mu(f)= 0$. Suppose that $f$ is a bad face. Then, $f$ sends $\frac{1}{5}$ to five incident $3$-vertices. So, $\mu(f)= 0$.
    \item If $f$ is a $6$-face: By Lemma \ref{l4}, we deduce that $f$ is incident to at most one $2$-vertex. Moreover, by Lemma \ref{l14}, we deduce that such a $2$-vertex is a good $2$-vertex. 
    
    Suppose that $f$ is incident to one $2$-vertex. Since that neighbors of a $2$-vertex are of degree $4$ by Lemma \ref{l2}, $f$ is incident to at most three $3$-vertices. Moreover, since $G$ is $4$-irregular, $f$ is incident to at least two $3$-vertices.  Suppose that $f$ is incident to three $3$-vertices. By Lemma \ref{l25}, these vertices are good $3$-vertices. So, $\mu(f)=0$ after $f$ sends $1$ to its incident $2$-vertex by R1 and $\frac{1}{3}$ to each incident $3$-vertex by R10.  Suppose that $f$ is incident to two $3$-vertices. By Lemma \ref{l26}, these vertices are not weak $3(1)$-vertices. So, $\mu (f)\ge 0$ after $f$ sends $1$ to its incident $2$-vertex by R1 and at most $\frac{1}{2}$ to each incident $3$-vertex by R7.

     Suppose that $f$ is not incident to any $2$-vertex. Suppose that $f$ is incident to at most three $3$-vertices. Then $\mu(f)\ge 0$ after $f$ sends at most $\frac{2}{3}$ to each incident $3$-vertex by R6. Suppose now that $f$ is incident to four $3$-vertices. By Lemma \ref{l24}, we deduce that $f$ is not incident to any weak $3(1)$-vertex. So, $v$ sends at most $\frac{1}{2}$ to each incident $3$-vertex by R7. Hence, $\mu(v)\ge d(f)-4-4 \times \frac{1}{2} \ge 0$.
     Suppose now that $f$ is incident to five $3$-vertices. By Lemma \ref{l27}, all incident $3$-vertices are good $3$-vertices. So, $f$ sends $\frac{1}{3}$ to each incident $3$-vertex by R10. Hence, we have $\mu(v) \ge 0$.
     Suppose now that $f$ is incident to six $3$-vertices. By Lemma \ref{l11}, we deduce that all incident vertices are good $3$-vertices. So, $f$ sends $\frac{1}{3}$ to each incident vertex by R10. Hence, we have $\mu(v) =0$.
     \item If $f$ is a $7$-face: By Lemma \ref{l4}, we deduce that $f$ is incident to at most one $2$-vertex. Such a vertex is a good $2$-vertex by \Cref{l15}. If $f$ is incident to a $2$-vertex, we deduce that $f$ is incident at most two $3(1)$-vertices by Lemma \ref{l7}. In such a case, we deduce by \Cref{l8} that $f$ is incident to three  $3$-vertices. 
     By Lemma \ref{l8}, we deduce that $f$ contains at most four  $3(1)$-vertices and in such a case the remaining incident vertices are of degree $4$. Note that if all incident vertices are of degree $3$, we deduce by \Cref{l11} that all incident vertices are good $3$-vertices.   Hence, the worst case is when $f$ is incident to a $2$-vertex, two weak $3(1)$-vertices, and one semi-bad $3$-vertex. So, $\mu(f)\ge 7-4-1 -2\times \frac{2}{3} -\frac{5}{12}\ge 0$ after $f$ sends $1$ to its incident $2$-vertex by R1, $\frac{2}{3}$ to each incident weak $3(1)$-vertex by R6, and $\frac{5}{12}$ to its incident semi-bad $3$-vertex by R9.
     \item If $f$ is an $8$-face:  By Lemma \ref{l4}, we deduce that $f$ is incident to at most two vertices of degree $2$.
     The worst case is when $f$ is incident to two bad $2$-vertices. Suppose that $f$ is incident to two bad $2$-vertices. Then, by \Cref{l28}, $f$ is incident to exactly two $3$-vertices that are good $3$-vertices. Thus, $\mu(f) \ge 8-4-2\times \frac{5}{3}-2\times \frac{1}{3}\ge 0$ after $f$ sends $\frac{5}{3}$ to each incident $2$-vertex by R1 and $\frac{1}{3}$ to each incident $3$-vertex by R10.
     \item If $f$ is a $9^+$-face: Let $k=d(f)$ such that $k \ge 9$. Let $r$ be the remainder of $\frac{k}{4}$. By Lemma \ref{l4}, we deduce that $f$ is incident to at most $\lfloor \frac{k}{4}\rfloor$ $2$-vertices. 
     In such a case, we deduce that $f$ is incident to $\lfloor\frac{k}{2}\rfloor$ $4$-vertices by \Cref{l2}. The worst case occurs when $f$ is incident to $\lfloor \frac{k}{4} \rfloor$ bad $2$-vertices.
     
     Suppose that $f$ is incident to $\lfloor \frac{k}{4} \rfloor$ bad $2$-vertices. Then, $f$ is incident to at most $(\lfloor \frac{k}{4}\rfloor+r)$ $3$-vertices. In this case, by Lemma \ref{l7} and \ref{l8}, we deduce that none of these $3(1)$-vertices. So, the worst case is when $f$ is incident to $(\lfloor \frac{k}{4}\rfloor+r)$ semi-bad $3$-vertices.
     Thus, $\mu(f) \ge k-4-\lfloor \frac{k}{4}\rfloor \times \frac{5}{3}-\frac{5}{12} \times (\lfloor \frac{k}{4}\rfloor +r)\ge 0$ after $f$ sends $\frac{5}{3}$ to each incident $2$-vertex by R2 and  $\frac{5}{12}$ to each semi-bad $3$-vertex by R9.
     \item If  $v$ is a $2$-vertex: Suppose that $v$ is a good $2$-vertex. Then, $v$ receives $1$ from each incident face by R1. Hence, $\mu(v)=0$. Suppose now that $v$ is a bad $2$-vertex. Then, $v$ is incident to an $8^+$-face by Lemma \ref{l14} and \ref{l15}. By Lemma \ref{l2}, we deduce that the $5$-face incident to $v$ is a good $5$-face. So, $v$ receives $\frac{1}{3}$ from its incident $5$-face by R3 and $\frac{5}{3}$ from its incident $8^+$-face by R2. So, $\mu(v)=0$.
     \item If $v$ is a $3$-vertex: 

     Suppose that $v$ is a $3(1)$-vertex. Suppose that $v$ is a strong $3(1)$-vertex. Then, $v$ receives $\frac{1}{2}$ from each incident $6^+$-face by R7. Hence, $\mu(v)=0$. Suppose now that $v$ is a weak $3(1)$-vertex. By Lemma \ref{l19}, $v$ is  incident to a good $5$-face. Then, $v$ receives $\frac{1}{3}$ from its incident $5$-face by R3 and $\frac{2}{3}$ from its incident $6^+$-face by R6. Hence, $\mu(v)=0$.

     Suppose now that $v$ is not incident to any $4$-face. Suppose that $v$ is a bad $3$-vertex. By Lemma \ref{l17}, $v$ is incident to two $6^+$-faces. Then, $v$ receives $\frac{1}{5}$ from its incident 5-face by R5 and $\frac{2}{5}$ from each incident $6^+$-face by R8. Hence, $\mu(v)=0$. Suppose that $v$ is a semi-bad $3$-vertex. By Lemma \ref{l18} and \ref{l20}, $v$ is incident to at least one $6^+$-face. By Lemma \ref{l11} and \ref{l16}, we deduce that $v$ is not incident to any other semi-bad or bad $5$-face. Then, $v$ receives $\frac{1}{4}$ from its incident semi-bad $5$-face by R4, $\frac{5}{12}$ from its incident $6^+$-face by R9, and at least $\frac{1}{3}$ from its third incident face. Hence, $\mu(v)\ge 0$. 
    Suppose that $v$ is a good $3$-vertex. Then, $v$ receives $\frac{1}{3}$ from each incident face. Hence, $\mu(v)\ge 0$. 
     
     \item If $v$ is a $4$-vertex: $v$ does not receive or send any charge. Hence, $\mu(v)=0$.
    
\end{itemize}

\vspace{0.5cm}
\textbf{Acknowledgment:} I would like to express my heartfelt appreciation to Dr. Maidoun Mortada for her continuous support and guidance in completing this work.

\end{document}